\documentclass[11pt,a4paper]{article}
\usepackage{amsmath}\usepackage{epsf,amsfonts,amsthm}\usepackage{amscd,amssymb}
\usepackage{xcolor,epic,eepic}\usepackage{epsfig}
\usepackage{fontenc,indentfirst, delarray,amsfonts,amsmath,amssymb}
\usepackage{rotating}
\usepackage{mathdots}
\usepackage[T1]{fontenc}
\usepackage[matrix,arrow,curve]{xy}
\usepackage{amsmath}
\usepackage{amssymb}
\usepackage{amsthm}
\usepackage{amscd}
\usepackage{amsfonts}
\usepackage{graphicx}
\usepackage{fancyhdr}
\usepackage{dsfont,texdraw}
\usepackage{amsmath}\usepackage{epsf,amsfonts,amsthm}\usepackage{amscd,amssymb}
\usepackage{xcolor,epic,eepic}\usepackage{epsfig}
\usepackage{fontenc,indentfirst, delarray,amsfonts,amsmath,amssymb}
\usepackage{rotating}
\usepackage{amscd}
\usepackage{amsmath}
\usepackage{amsthm}
\usepackage{mathrsfs}
\usepackage{latexsym}
\usepackage{amssymb}
\usepackage{amsfonts}
\theoremstyle{plain}

\usepackage{tikz}
\usetikzlibrary{arrows,chains,matrix,positioning,scopes,snakes}
\makeatletter
\tikzset{join/.code=\tikzset{after node path={%
\ifx\tikzchainprevious\pgfutil@empty\else(\tikzchainprevious)%
edge[every join]#1(\tikzchaincurrent)\fi}}}
\makeatother

\tikzset{>=stealth',every on chain/.append style={join},
         every join/.style={->}}
\tikzset{
    >=stealth',
    punkt/.style={
           rectangle,
           rounded corners,
           draw=black, very thick,
           text width=6.5em,
           minimum height=2em,
           text centered},
    pil/.style={
           ->,
           thick,
           shorten <=2pt,
           shorten >=2pt,}
}

\usepackage[T1]{fontenc}
\newcommand{\bee}{\begin{enumerate}}
\newcommand{\eee}{\end{enumerate}}
\newcommand{\benn}{\begin{equation*}}
\newcommand{\eenn}{\end{equation*}}
\newcommand{\be}{\begin{equation}}
\newcommand{\ee}{\end{equation}}
\newcommand{\bean}{\begin{eqnarray}}
\newcommand{\eean}{\end{eqnarray}}
\newcommand{\bea}{\begin{eqnarray*}}
\newcommand{\eea}{\end{eqnarray*}}
\newcommand{\w}{\wedge}

\newcommand{\p}{\partial}

\newcommand{\ra}{\rangle}

\newcommand{\Ci}{C^{\infty}}

\newcommand{\N}{\mathbb{N}}
\newcommand{\Z}{\mathbb{Z}}
\newcommand{\R}{\mathbb{R}}
\newcommand{\K}{\mathbb{K}}

\newcommand{\Q}{\mathbb{Q}}

\newcommand{\lp}{\left(}
\newcommand{\rp}{\right)}

\newcommand{\op}[1]{\!\!\mathop{\rm ~#1}\nolimits}

\newcommand{\mbi}{\mathbb{I}}

\mathchardef\za="710B  
\mathchardef\zb="710C  
\mathchardef\zg="710D  
\mathchardef\zd="710E  
\mathchardef\zve="710F 
\mathchardef\zz="7110  
\mathchardef\zh="7111  
\mathchardef\zy="7112 

\mathchardef\zi="7113  
\mathchardef\zk="7114  
\mathchardef\zl="7115  
\mathchardef\zm="7116  
\mathchardef\zn="7117  
\mathchardef\zx="7118  
\mathchardef\zp="7119  
\mathchardef\zr="711A  
\mathchardef\zs="711B  
\mathchardef\zt="711C  
\mathchardef\zu="711D  
\mathchardef\zf="711E 
\mathchardef\zq="711F  
\mathchardef\zc="7120  
\mathchardef\zw="7121  
\mathchardef\ze="7122  
\mathchardef\zvy="7123  
\mathchardef\zvw="7124  
\mathchardef\zvr="7125 
\mathchardef\zvs="7126 
\mathchardef\zvf="7127  
\mathchardef\zG="7000  
\mathchardef\zD="7001  
\mathchardef\zY="7002  
\mathchardef\zL="7003  
\mathchardef\zX="7004  
\mathchardef\zP="7005  
\mathchardef\zS="7006  
\mathchardef\zU="7007  
\mathchardef\zF="7008  
\mathchardef\zW="700A  

\newcommand{\cyclic}{\mathop{\kern0.9ex{{+}
\kern-2.15ex\raise-.25ex\hbox{\Large\hbox{$\circlearrowright$}}}}\limits}

\newcommand{\cE}{{\cal E}}
 \newcommand{\cS}{{\cal S}}

 \newcommand{\cH}{{\cal H}}
 \newcommand{\cP}{{\cal P}}
 
 \newcommand{\cA}{{\cal A}}
 \newcommand{\cM}{{\cal M}}
 \newcommand{\cD}{{\cal D}}
 \newcommand{\cO}{{\cal O}}

 \newcommand{\cQ}{{\cal Q}}
 \newcommand{\cI}{{\cal I}}

\newtheorem{rem}{Remark}
\newtheorem{theo}{Theorem}
\newtheorem{prop}{Proposition}
\newtheorem{lem}{Lemma}
\newtheorem{cor}{Corollary}

\newtheorem{defi}{Definition}

\newcommand{\h}{\op{Hom}}
\newcommand{\0}{\otimes}
\newcommand{\tc}{{\tt C}}
\newcommand{\id}{\op{id}}
\newcommand{\coker}{\op{coker}}

\DeclareMathAlphabet{\mathpzc}{OT1}{pzc}{m}{it}

 \newcommand{\colim}{\op{colim}}

\begin{document}
\title{\bf Model structure on differential graded commutative algebras over the ring of differential operators}
\date{}
\author{Gennaro di Brino, Damjan Pi\v{s}talo, and Norbert Poncin\footnote{University of Luxembourg, Mathematics Research Unit, 1359 Luxembourg City, Luxembourg, gennaro.dibrino@gmail.com, damjan.pistalo@uni.lu, norbert.poncin@uni.lu}}

\maketitle

\begin{abstract} We construct a cofibrantly generated model structure on the category of differential non-negatively graded quasi-coherent commutative $\cD_X$-algebras, where $\cD_X$ is the sheaf of differential operators of a smooth affine algebraic variety $X$. 
This article is the first of a series of works 
on a derived $\cD$-geometric approach to the Batalin-Vilkovisky formalism.

\end{abstract}

\vspace{2mm} \noindent {\bf MSC 2010}: 18G55, 16E45, 35A27, 32C38, 16S32

\thispagestyle{empty}

\section{Introduction}

The solution functor of a system of {\it linear} {\small PDE}-s $D\cdot m=0$ is a functor $\op{Sol}:{\tt Mod}(\cal D)\to {\tt Set}$ defined on the category of modules over the ring $\cD$ of (linear) differential operators of a suitable base space: for $D\in\cD$ and $M\in{\tt Mod}(\cD)$, we have $$\op{Sol}(M)=\{m\in M: D\cdot m=0\}\;.$$ For a system of {\it polynomial} {\small PDE}-s, we get (locally) a functor $\op{Sol}:{\tt Alg}(\cD)\to {\tt Set}$ defined on the category of $\cD$-algebras, i.e., commutative monoids in ${\tt Mod}(\cD)$. Just as the solutions of a system of polynomial algebraic equations lead to the concept of algebraic variety, the solutions of a system of {\it nonlinear} {\small PDE}-s are related to diffieties, to $\cal D$-schemes, or to locally representable sheaves $\op{Sol}:{\tt Alg}({\cal D})\to {\tt Set}$. To allow for still more general spaces, sheaves ${\tt Alg}({\cal D})\to {\tt SSet}$ valued in simplicial sets, or sheaves ${\tt DGAlg}({\cal D})\to {\tt SSet}$ on (the opposite of) the category ${\tt DGAlg}({\cal D})$ of differential graded $\cal D$-algebras have to be considered. The latter spaces are referred to as derived ${\cal D}$-stacks. Derived Algebraic $\cal D$-Geometry is expected to be the proper framework for a coordinate-free treatment of the `space of solutions of nonlinear {\small PDE}-s (modulo symmetries)', as well as of the Batalin-Vilkovisky formalism ({\small BV}) for gauge theories. The present paper is the first of a series on covariant derived $\cD$-geometric {\small BV}. The sheaf condition for functors ${\tt DGAlg}({\cal D})\to {\tt SSet}$ appears as the fibrant object condition with respect to a model structure ({\small MS}) on the category of these functors. The appropriate {\small MS} uses both, the model structure on $\tt SSet$ and the model structure on ${\tt DGAlg}(\cD)$. Moreover, the $\cD$-geometric counterpart of an algebra $\Ci(\zS)$ of on-shell functions turns out to be an algebra $A\in{\tt Alg}(\cD)\subset {\tt DGAlg}(\cD)$, so that the Koszul-Tate resolution of $\Ci(\zS)$ corresponds to a suitable cofibrant replacement of $A$ \cite{PP}.\medskip


In this first article, we describe a cofibrantly generated model structure on the category ${\tt DGAlg}(\cD)$ of differential non-negatively graded quasi-coherent commutative algebras over the sheaf $\cD$ of differential operators of a smooth affine algebraic variety $X$. This restriction on the underlying space $X$ allows to substitute global sections to sheaves, i.e., to consider the category of differential non-negatively graded commutative algebras over the ring $\cD(X)$ of global sections of $\cD$. The mentioned model structure is constructed, via Quillen's transfer theorem, from the cofibrantly generated projective model structure on the category ${\tt DGMod}(\cD)$ of differential non-negatively graded $\cD(X)$-modules. The latter is obtained from results of \cite{GS} and \cite{Hov}. Since, in contrast with \cite{GS,Hov}, we work over the special (sheaf of) noncommutative ring(s) of differential operators, a careful analysis is needed and some local subtleties cannot be avoided. Further, our restriction to affine varieties is not merely a comfort solution: the existence of a projective model structure requires that the underlying category have enough projectives -- this is in general not the case for a category of sheaves over a not necessarily affine scheme. Eventually, although results that hold over an affine base are of interest by themselves, we also expect them to provide insight into the structure of the main ingredients and hope that the fundamental aspects of the latter still make sense for arbitrary smooth schemes.\medskip

The paper is organized as follows:
\tableofcontents

\section{Conventions and notation}

According to the anglo-saxon nomenclature, we consider the number 0 as being neither positive, nor negative.\medskip

All the rings used in this text are implicitly assumed to be unital.\medskip

In most parts of our paper, the underlying space is a smooth affine algebraic variety.

\section{Sheaves of modules}\label{ShMod}

\newcommand{\cR}{{\cal R}}
\newcommand{\cF}{{\cal F}}

Let $\tt Top$ be the category of topological spaces and, for $X\in \tt Top$, let ${\tt Open}_X$ be the category of open subsets of $X$. If $\cR_X$ is a sheaf of rings, a {\bf left $\cR_X$-module} is a {\it sheaf $\cP_X$, such that, for each $U\in{\tt Open}_X$, $\cP_X(U)$ is an $\cR_X(U)$-module, and the $\cR_X(U)$-actions are compatible with the restrictions}. We denote by ${\tt Mod}(\cR_X)$ the Abelian category of $\cR_X$-modules and of their (naturally defined) morphisms.\medskip

In the following, we omit subscript $X$ if no confusion arises.\medskip

If $\cP,\cQ\in {\tt Mod}(\cR)$, the (internal) Hom ${\cH}om_{\cR}(\cP,\cQ)$ is the sheaf of Abelian groups (of $\cR$-modules, i.e., is the element of ${\tt Mod}(\cR)$, if $\cR$ is commutative) that is defined by \be{\cH}om_{\cR}(\cP,\cQ)(U):=\op{Hom}_{\cR|_U}(\cP|_U,\cQ|_U)\;,\label{HomSh}\ee $U\in{\tt Open}_X$. The {\small RHS} is made of the morphisms of (pre)sheaves of $\cR|_U$-modules, i.e., of the families $\zf_V:\cP(V)\to \cQ(V)$, $V\in{\tt Open}_U$, of $\cR(V)$-linear maps that commute with restrictions. Note that ${\cH}om_{\cR}(\cP,\cQ)$ is a sheaf of Abelian groups, whereas $\op{Hom}_{\cR}(\cP,\cQ)$ is the Abelian group of morphisms of (pre)sheaves of $\cR$-modules. We thus obtain a bi-functor \be\label{HomFun}{\cal H}om_{\cR}(\bullet,\bullet): ({\tt Mod}(\cR))^{\op{op}}\times {\tt Mod}(\cR)\to {\tt Sh}(X)\;,\ee valued in the category of sheaves of Abelian groups, which is left exact in both arguments. \medskip

Further, if $\cP\in{\tt Mod}(\cR^{\op{op}})$ and $\cQ\in{\tt Mod}(\cR)$, we denote by $\cP\otimes_\cR\cQ$ the sheaf of Abelian groups (of $\cR$-modules, if $\cR$ is commutative) associated to the presheaf \be\label{TensSh}(\cP\oslash_\cR\cQ)(U):=\cP(U)\otimes_{\cR(U)}\cQ(U)\;,\ee $U\in{\tt Open}_X$. The bi-functor \be\label{TensFun}\bullet\0_{\cR}\bullet:{\tt Mod}(\cR^{\op{op}})\times{\tt Mod}(\cR)\to {\tt Sh}(X)\;\ee is right exact in its two arguments.\medskip

If $\cS$ is a sheaf of commutative rings and $\cR$ a sheaf of rings, and if $\cS\to\cR$ is a morphism of sheafs of rings, whose image is contained in the center of $\cR$, we say that $\cR$ is a sheaf of $\cS$-algebras. Remark that, in this case, the above functors ${\cH}om_{\cR}(\bullet,\bullet)$ and $\bullet\0_\cR\bullet$ are valued in ${\tt Mod}(\cS)$.

\section{$\cD$-modules and $\cD$-algebras}

Depending on the author(s), the concept of $\cD$-module is considered over a base space $X$ that is a finite-dimensional smooth \cite{Cos} or complex \cite{KS} manifold, or a smooth algebraic variety \cite{HTT} or scheme \cite{BD}, over a fixed base field $\K$ of characteristic zero. We denote by $\cO_X$ (resp., $\Theta_X$, $\cD_X$) the sheaf of functions (resp., vector fields, differential operators acting on functions) of $X$, and take an interest in the category ${\tt Mod}(\cO_X)$ (resp., ${\tt Mod}(\cD_X)$) of $\cO_X$-modules (resp., $\cD_X$-modules).\medskip

Sometimes a (sheaf of) $\cD_X$-module(s) is systematically required to be {\it coherent} or {\it quasi-coherent} as (sheaf of) $\cO_X$-module(s). In this text, we will explicitly mention such extra assumptions.

\subsection{Construction of $\cD$-modules from $\cO$-modules}\label{D-ModulesAlgebras}

It is worth recalling the following

\begin{prop}\label{DModFlatConnSh} Let $\cM_X$ be an $\cO_X$-module. A left $\,\cD_X$-module structure on $\cM_X$ that extends its $\cO_X$-module structure is equivalent to a $\K$-linear morphism
$$\nabla :  \Theta_X \to {\cal E}nd_\K(\cM_X)\;,$$ such that, for all $f\in\cO_X$, $\zy,\zy'\in\Theta_X$, and all $m\in\cM_X$,
\begin{enumerate}
\item $\nabla_{f\zy}\,m = f\cdot\nabla_\zy m\,,$
\item $\nabla_\zy(f\cdot m)=f\cdot\nabla_\zy m+\zy(f)\cdot m\,,$
\item $\nabla_{[\zy,\zy']}m=[\nabla_\zy,\nabla_{\zy'}]m\,.$
\end{enumerate}
\end{prop}\medskip

In the sequel, we omit again subscript $X$, whenever possible.\medskip

In Proposition \ref{DModFlatConnSh}, the target ${\cE}nd_\K(\cM)$ is interpreted in the sense of Equation (\ref{HomSh}), and $\nabla$ is viewed as a morphism of sheaves of $\K$-vector spaces. Hence, $\nabla$ is a family $\nabla^U$, $U\in{\tt Open}_X$, of $\K$-linear maps that commute with restrictions, and $\nabla^U_{\zy_U}$, $\zy_U\in\Theta(U)$, is a family $(\nabla^U_{\zy_U})_V$, $V\in{\tt Open}_U$, of $\K$-linear maps that commute with restrictions. It follows that $\lp\nabla^U_{\zy_U}m_U\rp|_V=\nabla^V_{\zy_U|_V}m_U|_V\,$, with self-explaining notation: {\it the concept of sheaf morphism captures the locality of the connection $\nabla$ with respect to both arguments}.\medskip

Further, the requirement that the conditions (1) -- (3) be satisfied for all $f\in\cO$, $\zy,\zy'\in\Theta$, and $m\in\cM$, means that they must hold for any $U\in{\tt Open}_X$ and all $f_U\in\cO(U)$, $\zy_U,\zy_U'\in\Theta(U)$, and $m_U\in\cM(U)$.\medskip


We now detailed notation used in Proposition \ref{DModFlatConnSh}. An explanation of the underlying idea of this proposition can be found in Appendix \ref{D-modules}.

\subsection{Closed symmetric monoidal structure on ${\tt Mod}(\cD)$}

\newcommand{\cL}{{\cal L}}
\newcommand{\cN}{{\cal N}}

If we apply the Hom bi-functor (resp., the tensor product bi-functor) over $\cD$ (see (\ref{HomFun}) (resp., see (\ref{TensFun}))) to two left $\cD$-modules (resp., a right and a left $\cD$-module), we get only a (sheaf of) $\K$-vector space(s) (see remark at the end of Section \ref{ShMod}). {\it The good concept is the Hom bi-functor (resp., the tensor product bi-functor) over $\cO$}. Indeed, if $\cP,\cQ\in{\tt Mod}(\cD_X)\subset {\tt Mod}(\cO_X)$, the Hom sheaf ${\cal H}om_{\cO_X}(\cP,\cQ)$ (resp., the tensor product sheaf $\cP\0_{\cO_X} \cQ$) is a sheaf of $\cO_X$-modules. To define on this $\cO_X$-module, an extending left $\cD_X$-module structure, it suffices, as easily checked, to define the action of $\theta\in\Theta_X$ on $\zf\in{\cH}om_{\cO_X}(\cP,\cQ)$, for any $p\in\cP$, by \be\label{HomDMod}(\nabla_\theta\zf)(p)=\nabla_\theta(\zf(p))-\zf(\nabla_\theta p)\;\ee ( resp., on $p\0 q$, $p\in\cP, q\in \cQ$, by \be\label{TensDMod}\nabla_\theta(p\0 q)=(\nabla_\theta p)\0 q+p\0(\nabla_\theta q)\;)\;.\ee

The functor $${\cal H}om_{\cO_X}(\cP,\bullet):{\tt Mod}(\cD_X)\to {\tt Mod}(\cD_X)\;,$$ $\cP\in{\tt Mod}(\cD_X)$, is the right adjoint of the functor $$\bullet\0_{\cO_X}\cP:{\tt Mod}(\cD_X)\to {\tt Mod}(\cD_X)\;:$$ for any $\cN,\cP,\cQ\in{\tt Mod}(\cD_X)$, there is an isomorphism $${\cH}om_{\cD_X}(\cN\0_{\cO_X}\cP,\cQ)\ni f \mapsto (n\mapsto(p\mapsto f(n\0 p)))\in{\cH}om_{\cD_X}(\cN,{\cal H}om_{\cO_X}(\cP,\cQ))\;.$$
Hence, {\it the category $({\tt Mod}(\cD_X),\0_{\cO_X},\cO_X,{\cH}om_{\cO_X})$ is Abelian closed symmetric monoidal}. 
More details on $\cD$-modules can be found in \cite{KS, Scha, Schn}.

\begin{rem} In the following, the underlying space $X$ is a smooth algebraic variety over an algebraically closed field $\K$ of characteristic 0.\end{rem}

We denote by ${\tt qcMod}(\cO_X)$ (resp., ${\tt qcMod}(\cD_X)$) the Abelian category of quasi-coherent $\cO_X$-modules (resp., $\cD_X$-modules that are quasi-coherent as $\cO_X$-modules \cite{HTT}). This category is a full subcategory of ${\tt Mod}(\cO_X)$ (resp., ${\tt Mod}(\cD_X)$). Since further the tensor product of two quasi-coherent $\cO_X$-modules (resp., $\cO_X$-quasi-coherent $\cD_X$-modules) is again of this type, and since $\cO_X\in{\tt qcMod}(O_X)$ (resp., $\cO_X\in {\tt qcMod}(\cD_X)$), the category $({\tt qcMod}(\cO_X),\0_{\cO_X},\cO_X)$ (resp., $({\tt qcMod}(\cD_X),\0_{\cO_X},\cO_X)$) is a symmetric monoidal subcategory of $({\tt Mod}(\cO_X),\0_{\cO_X},\cO_X)$ (resp., $({\tt Mod}(\cD_X),\0_{\cO_X},\cO_X)$). For additional information on coherent and quasi-coherent modules over a ringed space, we refer to Appendix \ref{FinCondShMod}.

\subsection{Commutative $\cD$-algebras}

A $\cD_X$-algebra is a commutative monoid in the symmetric monoidal category ${\tt Mod}(\cD_X)$. More explicitly, a $\cD_X$-algebra is a $\cD_X$-module $\cA,$ together with $\cD_X$-linear maps $$\zm:\cA\0_{\cO_X}\cA\to \cA\quad\text{and}\quad \zi:\cO_X\to \cA\;,$$ which respect the usual associativity, unitality, and commutativity constraints. This means exactly that $\cA$ is a commutative associative unital $\cO_X$-algebra, which is endowed with a flat connection $\nabla$ -- see Proposition \ref{DModFlatConnSh} -- such that vector fields $\zy$ act as derivations $\nabla_\zy$. Indeed, when omitting the latter requirement, we forget the linearity of $\zm$ and $\zi$ with respect to the action of vector fields. Let us translate the $\Theta_X$-linearity of $\zm$. If $\zy\in\Theta_X,$ $a,a'\in\cA$, and if $a\ast a':=\zm(a\0 a')$, we get \be \nabla_\theta(a\ast a')=\nabla_\theta(\zm(a\0 a'))=\zm((\nabla_\theta a)\0 a'+a\0 (\nabla_\theta a'))=(\nabla_\theta a)\ast a'+a\ast (\nabla_\theta a')\;.\label{LeibDAlg}\ee If we set now $1_\cA:=\zi(1)$, Equation (\ref{LeibDAlg}) shows that $\nabla_\zy (1_\cA)=0$. It is easily checked that the $\Theta_X$-linearity of $\zi$ does not encode any new information. Hence,

\begin{defi} A commutative {\bf $\cD_X$-algebra} is a commutative monoid in ${\tt Mod}(\cD_X)$, i.e., a commutative associative unital $\cO_X$-algebra that is endowed with a flat connection $\nabla$ such that $\nabla_\zy$, $\zy\in\Theta_X,$ is a derivation.\end{defi}

\section{Differential graded $\cD$-modules and differential graded $\cD$-algebras}

\subsection{Monoidal categorical equivalence between chain complexes of $\cD_X$-modules and of $\cD_X(X)$-modules}\label{MonEquShMod}

Note first that any equivalence $F:{\tt C} \rightleftarrows {\tt D}:G$ between Abelian categories is exact. To see, for instance, that $F$ is an exact functor, let $0\to C'\to C\to C''\to 0$ be an exact sequence in $\tt C$. Exactness means that at each spot $\ker g=\op{im} f=\ker(\op{coker} f)$, where $f$ (resp., $g$) is the incoming (resp., outgoing) morphism. In other words, exactness means that at each spot $\lim \mathbf{G}=\lim(\op{colim}\mathbf{F})$, for some diagrams $\mathbf{G}:J\to\tt C$ and $\mathbf{F}:I\to \tt C$. However, in view of the equivalence, a functor $\mathbf{H}:K\to \tt C$ has the (co)limit $L$ if and only if the functor $F\mathbf{H}:K\to \tt D$ has the (co)limit $F(L)$. Consider now the $\tt D$-sequence $0\to F(C')\to F(C)\to F(C'')\to 0$. The kernels $\ker F(g)$ and $\ker(\op{coker}F(f))$ are the limits $\lim F\mathbf{G}=F(\lim \mathbf{G})$ and $\lim(\op{colim}F\mathbf{F})=F(\lim(\op{colim\, \mathbf{F}}))$, respectively, so that the considered kernels coincide and the $\tt D$-sequence is exact.\medskip

On the other hand, if $F:{\tt C} \rightleftarrows {\tt D}:G$ is an equivalence between monoidal categories, and if one of the functors $F$ or $G$ is strongly monoidal, then the other is strongly monoidal as well \cite{Hopf}. For instance, if $G$ is strongly monoidal, we have $$F(C\0 C')\simeq F(G(F(C))\0 G(F(C')))\simeq F(G(F(C)\0 F(C')))\simeq F(C)\0F(C')\;,$$ and, if $I\in\tt C$ and $J\in\tt D$ are the monoidal units, $$F(I)\simeq F(G(J))\simeq J\;.$$


It is well-known, see (\ref{ShVsSectqcMod}), that, for any affine algebraic variety $X$, we have the equivalence \be\label{ShVsSectqcMod1}\zG(X,\bullet):{\tt qcMod}(\cO_X)\to {\tt Mod}(\cO_X(X)):\widetilde\bullet\;\ee between Abelian symmetric monoidal categories, where $\widetilde\bullet$ is isomorphic to $\cO_X\0_{\cO_X(X)}\bullet\,$. Since the latter is obviously strongly monoidal, both functors, $\zG(X,\bullet)$ and $\widetilde\bullet\,$, are exact and strongly monoidal.\medskip


Similarly,

\begin{prop}\label{MonoidEquiv1} If $X$ is a smooth affine algebraic variety, its global section functor $\zG(X,\bullet)$ yields an equivalence \be\label{ShVsSectqcModDMon}\zG(X,\bullet):({\tt qcMod}(\cD_X),\0_{\cO_X},\cO_X)\to ({\tt Mod}(\cD_X(X)),\0_{\cO_X(X)},\cO_X(X))\;\ee between Abelian symmetric monoidal categories, and it is exact and strongly monoidal.\end{prop}

\begin{proof} For the categorical equivalence, see \cite[Proposition 1.4.4]{HTT}. Exactness is now clear and it suffices to show that $\zG(X,\bullet)$ is strongly monoidal. We know that $\zG(X,\bullet)$ is strongly monoidal as functor between modules over functions, see (\ref{ShVsSectqcMod1}). Hence, if $\cP,\cQ\in{\tt qcMod}(\cD_X)$, then \be\label{MonoidGlobSect}\zG(X,\cP\0_{\cO_X}\cQ)\simeq \zG(X,\cP)\0_{\cO_X(X)}\zG(X,\cQ)\ee as $\cO_X(X)$-modules. To define the $\cD_X$-module structure on $\cP\0_{\cO_X}\cQ$, we defined a family of compatible $\cD_X(U)$-module structures on the $\lp\cP\0_{\cO_X}\cQ\rp(U)$, $U\in{\tt Open}_X$, by setting, for $\zy_U\in\Theta_X(U)\,,$ $$\nabla^U_{\zy_U}:\cP(U)\0_{\cO_X(U)}\cQ(U)\ni p\0 q\mapsto $$ $$(\nabla^U_{\zy_U}p)\0 q+p\0(\nabla^U_{\zy_U}q)\in \cP(U)\0_{\cO_X(U)}\cQ(U)\subset(\cP\0_{\cO_X}\cQ)(U)\;,$$ where the inclusion means that any section of the presheaf $\cP\oslash_{\cO_X}\cQ$ can be viewed as a section of its sheafification $\cP\0_{\cO_X}\cQ$ (see Equation (\ref{TensSh}), Equation (\ref{TensDMod}) and Appendix \ref{D-modules}). We then (implicitly) `extended' $\nabla^U_{\zy_U}$ from $\cP(U)\0_{\cO_X(U)}\cQ(U)$ to $(\cP\0_{\cO_X}\cQ)(U)$. In view of (\ref{MonoidGlobSect}), the action $\nabla^X$ of $\Theta_X(X)$ on $\cP(X)\0_{\cO_X(X)}\cQ(X)$ and $(\cP\0_{\cO_X}\cQ)(X)$ `coincide', and so do the $\cD_X(X)$-module structures of these modules. Eventually, the global section functor is strongly monoidal.
\end{proof}

\begin{rem} In the sequel, we work systematically over a smooth affine algebraic variety $X$ over an algebraically closed field $\K$ of characteristic 0.\end{rem}

Since the category ${\tt qcMod}(\cD_X)$ is Abelian symmetric monoidal, the category ${\tt DG_+qcMod}(\cD_X)$ of differential non-negatively graded $\cO_X$-quasi-coherent $\cD_X$-modules is Abelian and symmetric monoidal as well -- for the usual tensor product of chain complexes and chain maps -- . The unit of this tensor product is the chain complex $\cO_X$ concentrated in degree 0. The braiding $\zb:\cP_\bullet\0 \cQ_\bullet\to \cQ_\bullet\0 \cP_\bullet$ is given by $$\zb(p\0 q)=(-1)^{\tilde p\tilde q}q\0 p\;,$$ where `tilde' denotes the degree and where the sign is necessary to obtain a chain map. Let us also mention that the zero object of ${\tt DG_+qcMod}(\cD_X)$ is the chain complex $(\{0\},0)\,$.


\begin{prop}\label{MonoidEquiv2} If $X$ is a smooth affine algebraic variety, its global section functor induces an equivalence \be\label{ShVsSectDGqcModDMon}\zG(X,\bullet):({\tt DG_+qcMod}(\cD_X),\0_{\cO_X},\cO_X)\to ({\tt DG_+Mod}(\cD_X(X)),\0_{\cO_X(X)},\cO_X(X))\;\ee of Abelian symmetric monoidal categories, and is exact and strongly monoidal.\end{prop}

\begin{proof} We first show that the categories ${\tt DG_+qcMod}(\cD_X)$ and ${\tt DG_+Mod}(\cD_X(X))$ are equivalent, then that the equivalence is strongly monoidal.\medskip

Let $F=\zG(X,\bullet)$ and $G$ be quasi-inverse (additive) functors that implement the equivalence (\ref{ShVsSectqcModDMon}). They induce functors $\mathbf{F}$ and $\mathbf{G}$ between the corresponding categories of chain complexes. Moreover, the natural isomorphism $a:\id\Rightarrow G\circ F$ induces, for each chain complex $\cP_\bullet\in{\tt DG_+qcMod}(\cD_X)$, a chain isomorphism $\mathbf{a}_{\cP_\bullet}:\cP_\bullet\rightarrow (\mathbf{G\circ F})(\cP_\bullet)$, which is functorial in $\cP_\bullet\,$. Both, the chain morphism property of $\mathbf{a}_{\cP_\bullet}$ and the naturality of $\mathbf{a}$, are direct consequences of the naturality of $a$. Similarly, the natural isomorphism $b:F\circ G\Rightarrow \id$ induces a natural isomorphism $\mathbf{b}:\mathbf{F\circ G}\Rightarrow\id$, so that ${\tt DG_+qcMod}(\cD_X)$ and ${\tt DG_+Mod}(\cD_X(X))$ are actually equivalent categories.\medskip

It suffices now to check that Proposition \ref{MonoidEquiv1} implies that $\mathbf{F}$ is strongly monoidal. Let $(\cP_\bullet,d),(\cQ_\bullet,\zd)\in{\tt DG_+qcMod}(\cD_X)$: $$\cP_\bullet\0_{\cO_X}\cQ_\bullet:\ldots\longrightarrow \bigoplus_{k+\ell=n+1}\cP_k\0_{\cO_X}\cQ_\ell\stackrel{\p}{\longrightarrow}\bigoplus_{k+\ell=n}\cP_k\0_{\cO_X}\cQ_\ell\longrightarrow\ldots\;,$$ where $\p=d\0\id+\id\0\,\zd\,.$ Since $F:{\tt qcMod}(\cD_X)\to{\tt Mod}(\cD_X(X))$ is strongly monoidal and commutes with colimits (recall that $F$ is left adjoint of $G$ and that left adjoints commute with colimits), its application to the preceding sequence leads to $$\mathbf{F}(\cP_\bullet\0_{\cO_X}\cQ_\bullet):\ldots\longrightarrow \bigoplus_{k+\ell=n+1}F(\cP_k)\0_{\cO_X(X)}F(\cQ_\ell)\stackrel{F(\p)}{\longrightarrow}\bigoplus_{k+\ell=n}F(\cP_k)\0_{\cO_X(X)}F(\cQ_\ell)\longrightarrow\ldots\;,$$ with $F(\p)=F(d)\0\id+\id\0\,F(\zd)$. In other words, $\mathbf{F}(\cP_\bullet\0_{\cO_X}\cQ_\bullet)$ coincides up to isomorphism, as chain complex, with $\mathbf{F}(\cP_\bullet)\0_{\cO_X(X)}\mathbf{F}(\cQ_\bullet)$. The remaining requirements are readily checked. \end{proof}

\subsection{Differential graded $\cD_X$-algebras vs. differential graded $\cD_X(X)$-algebras}

The strongly monoidal functors $\mathbf{F}:{\tt DG_+qcMod}(\cD_X)\rightleftarrows{\tt DG_+Mod}(\cD_X(X)):\mathbf{G}$ yield an equivalence between the corresponding categories of commutative monoids:

\begin{cor} For any smooth affine variety $X$, there is an equivalence of categories \be\label{ShVsSectqcDAlg} \zG(X,\bullet): {\tt DG_+qcCAlg}(\cD_X)\rightarrow{\tt DG_+CAlg}(\cD_X(X))\;\ee between the category of differential graded quasi-coherent commutative $\cD_X$-algebras and the category of differential graded commutative $\cD_X(X)$-algebras.\end{cor}

The main goal of the present paper is to construct a model category structure on the {\small LHS} category. In view of the preceding corollary, it suffices to build this model structure on the {\small RHS} category. We thus deal in the sequel exclusively with this category of {\bf differential graded $\cD$-algebras}, where $\cD:=\cD_X(X)$, which we denote simply by $\tt DG\cD A$. Similarly, the objects of ${\tt DG_+Mod}(\cD_X(X))$ are termed {\bf differential graded $\cD$-modules} and their category is denoted by $\tt DG\cD M$.

\subsection{The category $\tt DG\cD A$}

In this subsection we describe the category $\tt DG\cD A$ and prove first properties.\medskip

Whereas $$\op{Hom}_\cD(P,Q)=\op{Hom}_{{\tt Mod}(\cD)}(P,Q)\;,$$ $P,Q\in{\tt Mod}(\cD)$, is a $\K$-vector space, the set $$\op{Hom}_{\cD{\tt A}}(A,B)=\op{Hom}_{{\tt CAlg}(\cD)}(A,B)\;,$$ $A,B\in{\tt CAlg}(\cD)$, is even not an Abelian group. Hence, there is no category of chain complexes over commutative $\cD$-algebras and the objects of $\tt DG\cD A$ are (probably useless to say) no chain complexes of algebras.\medskip

As explained above, a $\cD$-algebra is a commutative unital $\cO$-algebra, endowed with a (an extending) $\cD$-module structure, such that vector fields act by derivations. Analogously, a differential graded $\cD$-algebra is easily seen to be a differential graded commutative unital $\cO$-algebra (a graded $\cO$-module together with an $\cO$-bilinear degree respecting multiplication, which is associative, unital, and graded-commutative; this module comes with a square 0, degree $-1$, $\cO$-linear, graded derivation), which is also a differential graded $\cD$-module (for the same differential, grading, and $\cO$-action), such that vector fields act as non-graded derivations.

\begin{prop} A differential graded $\cD$-algebra is a differential graded commutative unital $\cO$-algebra, as well as a differential graded $\cD$-module, such that vector fields act as derivations. Further, the morphisms of $\tt DG\cD A$ are the morphisms of $\tt DG\cD M$ that respect the multiplications and units. \end{prop}

In fact:

\begin{prop} The category $\tt DG\cD A$ is symmetric monoidal for the tensor product of $\tt DG\cD M$ with values on objects that are promoted canonically from $\tt DG\cD M$ to $\tt DG\cD A$ and same values on morphisms. The tensor unit is $\cO$; the initial object $(\,$resp., terminal object$\,)$ is $\cO$ $(\,$resp., $\{0\}$$\,)$.\end{prop}

\begin{proof} Let $A_\bullet,B_\bullet\in\tt DG\cD A$. Consider homogeneous vectors $a\in A_{\tilde{a}}$, $a'\in A_{\tilde{a}'}$, $b\in B_{\tilde{b}}$, $b'\in B_{\tilde{b}'}$, such that $\tilde a+\tilde b=m$ and $\tilde a'+\tilde b'=n$. Endow now the tensor product $A_\bullet\0_\cO B_\bullet\in\tt DG\cD M$ with the multiplication $\star$ defined by \be\label{MultTensMod}(A_\bullet\0_\cO B_\bullet)_m\times (A_\bullet\0_\cO B_\bullet)_n\ni(a\0 b,a'\0 b')\mapsto $$ $$(a\0 b)\star(a'\0 b') = (-1)^{\tilde a'\tilde b}(a\star_A a')\0 (b\star_B b')\in (A_\bullet\0_\cO B_\bullet)_{m+n}\;,\ee where the multiplications of $A_\bullet$ and $B_\bullet$ are denoted by $\star_A$ and $\star_B$, respectively. The multiplication $\star$ equips $A_\bullet\0_\cO B_\bullet$ with a structure of differential graded $\cD$-algebra. Note also that the multiplication of $A_\bullet\in\tt DG\cD A$ is a $\tt DG\cD A$-morphism $\zm_A:A_\bullet\0_\cO A_\bullet\to A_\bullet\,$.\medskip

Further, the unit of the tensor product in $\tt DG\cD A$ is the unit $(\cO,0)$ of the tensor product in $\tt DG\cD M$.\medskip

Finally, let $A_\bullet,B_\bullet,C_\bullet,D_\bullet\in\tt DG\cD A$ and let $\zf:A_\bullet\to C_\bullet$ and $\psi:B_\bullet\to D_\bullet$ be two $\tt DG\cD A$-morphisms. Then the $\tt DG\cD M$-morphism $\zf\0\psi:A_\bullet\0_\cO B_\bullet\to C_\bullet\0_\cO D_\bullet$ is also a $\tt DG\cD A$-morphism.\medskip

All these claims (as well as all the additional requirements for a symmetric monoidal structure) are straightforwardly checked.\medskip

The initial and terminal objects in $\tt DG\cD A$ are the differential graded $\cD$-algebras $(\cO,0)$ and $(\{0\},0)$, respectively. As concerns the terminal object, this is the expected and easily verified result. The initial object however is not the same as the one in $\tt DG\cD M$. The problem with the initial object candidate $(\{0\},0)\,$, is that a $\tt DG\cD A$-morphism $\zf:(\{0\},0)\to (A_\bullet,d_A)$ has to map $0$ to $0_A$ and to $1_A$, what in only possible if $0_A=1_A$, i.e., if $A_\bullet=\{0\}\,$. As for $(\cO,0)$, the sole point to check is that the unique morphism $\zf:(\cO,0)\to (A_\bullet,d_A)$, which is necessarily defined by $$\zf(f)=\zf(f\cdot 1_\cO)=f\cdot \zf(1_\cO)=f\cdot 1_A\;,$$ is a $\tt DG\cD A$-morphism. For the latter, only $\cD$-linearity, i.e., $\Theta$-linearity, has to be checked. We get $$\zf(\nabla_\zy f)=\zf(\zy(f))=\zy(f)\cdot 1_A\;,$$ whereas $$\nabla_\zy(\zf(f))=\nabla_\zy(f\cdot 1_A)=\nabla_{\zy\circ f} 1_A=\zy(f)\cdot 1_A+\nabla_{f\circ \zy} 1_A=\zy(f)\cdot 1_A\;,$$ as in a differential graded $\cD$-algebra vector fields act as derivations and thus annihilate the unit.\end{proof}

Let us still mention the following

\begin{prop}\label{ProdTensProd} If $\zf:A_\bullet\to C_\bullet$ and $\psi:B_\bullet\to C_\bullet$ are $\tt DG\cD A$-morphisms, then $\chi:A_\bullet\0_\cO B_\bullet\to C_\bullet$, which is well-defined by $\chi(a\0 b)=\zf(a)\star_C \psi(b),$ is a $\tt DG\cD A$-morphism that restricts to $\zf$ (resp., $\psi$) on $A_\bullet$ (resp., $B_\bullet$).\end{prop}

\begin{proof} It suffices to observe that $\chi=\zm_C\circ(\zf\0 \psi)\,$.\end{proof}

\section{Finitely generated model structure on $\tt DG{\cD}M$}

All relevant information on model categories, small objects, and on cofibrantly and finitely generated model structures, can be found in Appendices \ref{ModCat}, \ref{Small}, and \ref{CofGenModCat}.\medskip

Let us recall that $\tt DG\cD M$ is the category ${\tt Ch}_+(\cD)$ of non-negatively graded chain complexes of left modules over the non-commutative unital ring $\cD=\cD_X(X)$ of differential operators of a smooth affine algebraic variety $X$. The remaining part of this section actually holds for any not necessarily commutative unital ring $R$ and the corresponding category ${\tt Ch}_+(R)$. We will show that ${\tt Ch}_+(R)$ is a finitely (and thus cofibrantly) generated model category.\medskip

In fact, most of the familiar model categories are cofibrantly generated. For instance, in the model category $\tt SSet$ of simplicial sets, the generating cofibrations $I$ (resp., the generating trivial cofibrations $J$) are the canonical simplicial maps $\p\zD[n]\to \zD[n]$ from the boundaries of the standard simplicial $n$-simplices to these simplices (resp., the canonical maps $\zL^r[n]\to \zD[n]$ from the $r$-horns of the standard $n$-simplices, $0\le r\le n$, to these simplices). The generating cofibrations and trivial cofibrations of the model category $\tt Top$ of topological spaces -- which is Quillen equivalent to $\tt SSet$ -- are defined similarly. The homological situation is analogous to the topological and combinatorial ones. In the case of ${\tt Ch}_+(R)$, the set $I$ of generating cofibrations (resp., the set $J$ of generating trivial cofibrations) is made (roughly) of the maps $S^{n-1}\to D^n$ from the $(n-1)$-sphere to the $n$-disc (resp., of the maps $0\to D^n$). In fact, the $n$-disc $D^n$ is the chain complex \be\label{Disc}D^n_{\bullet}: \cdots \to 0\to 0\to \stackrel{(n)}{R} \to \stackrel{(n-1)}{R}\to 0\to \cdots\to \stackrel{(0)}{0}\;,\ee whereas the $n$-sphere $S^n$ is the chain complex \be\label{Sphere}S^n_\bullet: \cdots \to 0\to 0\to \stackrel{(n)}{R}\to 0\to \cdots\to \stackrel{(0)}{0}\;.\ee Definition (\ref{Disc}), in which the differential is necessarily the identity of $R$, is valid for $n\ge 1$. Definition (\ref{Sphere}) makes sense for $n\ge 0$. We extend the first (resp., second) definition to $n=0$ (resp., $n=-1$) by setting $D^0_\bullet:=S^0_\bullet$ (resp., $S^{-1}_\bullet:=0_\bullet$). The chain maps $S^{n-1}\to D^n$ are canonical (in degree $n-1$, they necessarily coincide with $\id_R$), and so are the chain maps $0\to D^n$. We now define the set $I$ (resp., $J$) by \be\label{GenCof} I=\{\iota_n: S^{n-1} \to D^n, n\ge 0\}\ee $(\,$resp., \be\label{GenTrivCof}J=\{\zeta_n: 0 \to D^n, n\ge 1\}\;)\;.\ee

\begin{theo}\label{FinGenModDGDM}
For any unital ring $R$, the category ${\tt Ch}_+(R)$ of non-negatively graded chain complexes of left $R$-modules is a finitely $(\,$and thus a cofibrantly$\,)$ generated model category $(\,$in the sense of \cite{GS} and in the sense of \cite{Hov}$\,)$, with $I$ as its generating set of cofibrations and $J$ as its generating set of trivial cofibrations. The weak equivalences are the chain maps that induce isomorphisms in homology, the cofibrations are the injective chain maps with degree-wise projective cokernel $(\,$projective object in ${\tt Mod}(R)$$\,)$, and the fibrations are the chain maps that are surjective in $(\,$strictly$\,)$ positive degrees. Further, the trivial cofibrations are the injective chain maps $i$ whose cokernel $\coker(i)$ is strongly projective as a chain complex $(\,$strongly projective object $\coker(i)$ in ${\tt Ch}_+(R)$, in the sense that, for any chain map $c:\coker(i)\to C$ and any chain map $p:D\to C$, there is a chain map $\ell:\coker(i)\to D$ such that $p\circ\ell=i$, if $p$ is surjective in $(\,$strictly$\,)$ positive degrees$\,)$.\end{theo}

\begin{proof}

The following proof uses the differences between the definitions of (cofibrantly generated) model categories given in \cite{DS}, \cite{GS}, and \cite{Hov}: we refer again to the Appendices \ref{ModCat}, \ref{Small}, and \ref{CofGenModCat}.\medskip

It is known that ${\tt Ch}_+(R)$, with the described weq-s, cofibrations, and fibrations is a model category (Theorem 7.2 in \cite{DS}). A model category in the sense of \cite{DS} contains all finite limits and colimits; the $\op{Cof}-\op{TrivFib}$ and $\op{TrivCof}-\op{Fib}$ factorizations are neither assumed to be functorial, nor, of course, to be chosen functorial factorizations. Moreover, we have $\op{Fib}=\op{RLP}(J)$ and $\op{TrivFib}=\op{RLP}(I)$ (Proposition 7.19 in \cite{DS}).\medskip

Note first that ${\tt Ch}_+(R)$ has all small limits and colimits, which are taken degree-wise.\medskip

Observe also that the domains and codomains $S^n$ ($n\ge 0$) and $D^n$ ($n\ge 1$) of the maps in $I$ and $J$ are bounded chain complexes of finitely presented $R$-modules (the involved modules are all equal to $R$). However, every bounded chain complex of finitely presented $R$-modules is $n$-small, $n\in\N$, relative to all chain maps (Lemma 2.3.2 in \cite{Hov}). Hence, the domains and codomains of $I$ and $J$ satisfy the smallness condition of a finitely generated model category, and are therefore small in the sense of the finite and transfinite definitions of a cofibrantly generated model category.\medskip

It thus follows from the Small Object Argument -- see Appendix \ref{CofGenModCat} -- that there exist in ${\tt Ch}_+(R)$ a functorial $\op{Cof}-\op{TrivFib}$ and a functorial $\op{TrivCof}-\op{Fib}$ factorization. Hence, the first part of Theorem \ref{FinGenModDGDM}.\medskip

As for the part on trivial cofibrations, its proof is the same as the proof of Lemma 2.2.11 in \cite{Hov}.\end{proof}

In view of Theorem \ref{FinGenModDGDM}, let us recall that any projective chain complex $(K,d)$ is degree-wise projective. Indeed, consider, for $n\ge 0$, an $R$-linear map $k_n:K_n\to N$ and a surjective $R$-linear map $p:M\to N$, and denote by $D^{n+1}(N)$ (resp., $D^{n+1}(M)$) the disc defined as in (\ref{Disc}), except that $R$ is replaced by $N$ (resp., $M$). Then there is a chain map $k:K\to D^{n+1}(N)$ (resp., a surjective chain map $\zp:D^{n+1}(M)\to D^{n+1}(N)$) that is zero in each degree, except in degree $n+1$ where it is $k_n\circ d_{n+1}$ (resp., $p$) and in degree $n$ where it is $k_n$ (resp., $p$). Since $(K,d)$ is projective as chain complex, there is a chain map $\ell:K\to D^{n+1}(M)$ such that $\zp\circ \ell =k$. In particular, $\ell_n:K_n\to M$ is $R$-linear and $p\circ \ell_n=k_n\,.$

\section{Finitely generated model structure on $\tt DG\cD A$}

\subsection{Adjoint functors between $\tt DG\cD M$ and $\tt DG\cD A$}\label{Adjunction}

We aim at transferring to $\tt DG\cD A$, the just described finitely generated model structure on $\tt DG\cD M$. Therefore, we need a pair of adjoint functors.

\begin{prop} The graded symmetric tensor algebra functor $\cS$ and the forgetful functor $\op{For}$ provide an adjoint pair $${\cal S}:{\tt DG\cD M}\rightleftarrows{\tt DG\cD A}:\op{For}\;$$ between the category of differential graded $\cD$-modules and the category of differential graded $\cD$-algebras.\end{prop}

\begin{proof} For any $M_\bullet\in{\tt DG\cD M}$, we get $$\0_\cO^\ast M_\bullet=\cO\oplus\bigoplus_{n\ge 1}M_\bullet^{\0_\cO n}\in{\tt DG\cD M}\;.$$ Moreover, $\0_\cO^\ast M_\bullet$ is the free associative unital $\cO$-algebra over the $\cO$-module $M_\bullet\,.$ When passing to graded symmetric tensors, we divide by the $\cO$-ideal $\cI$ generated by the elements of the type $m_k\0 m_\ell-(-1)^{k\ell}m_\ell\0 m_k$, where $m_k\in M_k$ and $m_\ell\in M_\ell\,$. This ideal is also a graded sub $\cD$-module that is stable for the differential, so that $\cI$ is actually a sub {\small DG} $\cD$-module. Therefore, the free graded symmetric unital $\cO$-algebra \be\label{Alg1}{\cal S}_\cO^\ast M_\bullet=\0_\cO^\ast M_\bullet/{\cal I}\ee is also a {\small DG} $\cD$-module. Of course, the graded symmetric tensor product $\odot$ of graded symmetric tensors $[S],[T]$ is defined by \be\label{Mult1}[S]\odot [T]=[S\0 T]\;.\ee Since the differential (resp., the $\cD$-action) on ${\cal S}_\cO^\ast M_\bullet$ is induced by that on $\0_\cO^\ast M_\bullet\,$, it is clear that the differential is a graded derivation of (resp., that vector fields act as derivations on) the graded symmetric tensor product. Hence, ${\cal S}_\cO^\ast M_\bullet\in {\tt DG\cD A}$. The definition of $\cS$ on morphisms is obvious.\medskip

We now prove that the functors $\op{For}$ and $\cal S$ are adjoint, i.e., that \be\label{Adjoint}\h_{\tt DG\cD A}({\cal S}_\cO^\ast M_\bullet,A_\bullet)\simeq \h_{\tt \tt DG\cD M}(M_\bullet,\op{For}A_\bullet)\;,\ee functorially in $M_\bullet\in{\tt DG\cD M}$ and $A_\bullet\in{\tt DG\cD A}\,$.

Since the inclusion $$i:M_\bullet \rightarrow {\cal S}_\cO^\ast M_\bullet=\cO_\bullet\oplus M_\bullet\oplus \bigoplus_{n\ge 2}S_\cO^n M_\bullet$$ is obviously a $\tt DG\cD M$-map, any $\tt DG\cD A$-map $\Phi:{\cal S}_\cO^\ast M_\bullet\to A_\bullet$ gives rise to a $\tt DG\cD M$-map $\Phi\circ i:M_\bullet\to \op{For}A_\bullet\,$.

Conversely, let $\zf:M_\bullet\to \op{For}A_\bullet$ be a $\tt DG\cD M$-map. Since ${\cal S}_\cO^\ast M_\bullet$ is free in the category $\tt GCA$ of graded commutative associative unital graded $\cO$-algebras, a $\tt GCA$-morphism is completely determined by its restriction to the graded $\cO$-module $M_\bullet\,$. Hence, the extension $\bar\zf:{\cal S}_\cO^\ast M_\bullet\to A_\bullet$ of $\zf$, defined by $\bar\zf(1_\cO)=1_{A}$ and by $$\bar\zf(m_1\odot\ldots\odot m_k)=\zf(m_1)\star_A\ldots\star_A\zf(m_k)\;,$$ is a $\tt GCA$-morphism. This extension is also a $\tt DG\cD A$-map, i.e., a $\tt DG\cD M$-map that respects the multiplications and the units, if it intertwines the differentials and is $\cD$-linear. These requirements, as well as functoriality, are straightforwardly checked.\end{proof}

Recall that a free object in a category $\tt D$ over an object $C$ in a category $\tt C$, such that there is a forgetful functor $\op{For}:\tt D\to C$, is a universal pair $(F(C),i)$, where $F(C)\in\tt D$ and $i\in\op{Hom}_{\tt C}(C,\op{For} F(C))\,$.

\begin{rem}\label{FreeDGDA} Equation (\ref{Adjoint}) means that ${\cal S}_\cO^\star M_\bullet$ is the {\bf free differential graded $\cD$-algebra} over the differential graded $\cD$-module $M_\bullet\,$.\end{rem}

A definition of $\cS_\cO^\ast M_\bullet$ via invariants can be found in Appendix \ref{InvCoinv}.

\subsection{Relative Sullivan $\cD$-algebras}\label{RSDA}

To transfer the model structure from $\tt DG\cD M$ to $\tt DG\cD A$, some preparation is necessary. In the present subsection, we define relative Sullivan algebras over the ring of differential operators. Background information on relative Sullivan algebras over a field can be found in Appendix \ref{RSKA}. \medskip

Let $V$ be a {\it free non-negatively graded $\cD$-module}. If we denote its basis by $(v_\za)_{\za\in J}$, we get $$V=\bigoplus_{\za\in J}\,\cD\cdot v_\za\;.$$ In almost all examples that occur in this text, the $v_\za$-s are formal generators and the elements $v\in V$ are formal linear combinations $v=\sum_{\za\in J} D^\za\cdot v_\za$, where only a finite number of coefficients $D^\za\in\cD$ are non-zero. We will further assume that $V$ is non-negatively graded, $$V_\bullet=\bigoplus_{\ell\in\N}V_\ell\;.$$ Often the basis vectors $v_\za$ have homogeneous degrees in $\N$ and the latter decomposition is induced by these degrees. In the next definition, we ask that the index set $J$ be a {\it well-ordered set}, see Appendix \ref{Small}. Even if the basis vectors have homogeneous degrees, we do a priori not require that this well-ordering $\le$ of $J$ be compatible with the degree $\deg$ in $V$, i.e., we do not suppose that \be\label{minimal}\za\le\zb \Rightarrow \deg v_\za\le\deg v_\zb\;.\ee

Examples of such free non-negatively graded $\cD$-modules are the $n$-sphere \be\label{Sphere2}S^n_\bullet: \cdots \to 0\to 0\to \cD\cdot 1_n\to 0\to \cdots\to 0\;\ee ($n\ge 0$) and the $n$-disc  \be\label{Disc2}D^n_{\bullet}: \cdots \to 0\to 0\to \cD\cdot \mathbb{I}_n \to \cD\cdot s^{-1}\mathbb{I}_n\to 0\to \cdots\to 0\;\ee ($n\ge 1\,$). Observe that, in view of future needs, we used different notation for the generator $1\in\cO$ of $\cD$ ($s^{-1}$ denotes the desuspension operator). Of course, if $D\in\cD$, we can identify $D\cdot 1_n, D\cdot \mathbb{I}_n, D\cdot s^{-1}\mathbb{I}_n$ with $D$, whenever no confusion arises.\medskip

If we endow $V$ with the differential $0$, we get $V\in\tt DG\cD M$, so that $\cS^\star_\cO V_\bullet\in\tt DG\cD A$, again with differential $0$. Let now $(A_\bullet,d_A)\in\tt DG\cD A$. The tensor product $A_\bullet\0_\cO\cS^\star_\cO V_\bullet$, where $A$ is considered, not with its original differential $d_A$, but with differential $0$, is a {\small DG$\cD$A} with differential $0$. In the following definition, we assume that this tensor product {\small G$\cD$A} is equipped with a differential $d$, which makes it an element $$(A_\bullet\0_\cO\cS^\star_\cO V_\bullet,d)\in\tt DG\cD A$$ that contains $(A_\bullet,d_A)$ as sub-{\small DG$\cD$A}. The point is here that $(A_\bullet,d_A)$ is a differential submodule of the tensor product differential module, but that usually the module $\cS^\star_\cO V_\bullet$ is not. The condition that $(A_\bullet,d_A)$ be a sub-{\small DG$\cD$A} can be rephrased by asking that the inclusion $$A_\bullet\ni a\mapsto a\0 1\in A_\bullet\0_\cO\cS^\star_\cO V_\bullet\;$$ be a $\tt DG\cD A$-morphism. This algebra morphism condition or subalgebra condition would be automatically satisfied, if the differential $d$ on $A_\bullet\0_\cO\cS^\star_\cO V_\bullet$ was obtained as \be\label{split}d=d_A\0 \id + \id\0 d_\cS\ee from the differential $d_A$ on $A_\bullet$ and a differential $d_\cS$ on $\cS^\star_\cO V_\bullet$. However, as mentioned, this is generally not the case.\medskip

We omit in the sequel $\bullet,$ $\star,$ as well as subscript $\cO$, provided clarity does not suffer hereof. Further, to avoid confusion, we sometimes substitute $\boxtimes$ to $\0$ to emphasize that the differential $d$ of $A\boxtimes\cS V$ is not necessarily obtained from the differential $d_A$ and a differential $d_\cS$.

\begin{defi} A {\bf relative Sullivan $\cD$-algebra} $(\,${\small RS$\cD\!$A}$\,)$ is a $\tt DG\cD A$-morphism
$$(A,d_A)\to(A\boxtimes \cS V,d)\;$$ that sends $a\in A$ to $a\0 1\in A\boxtimes \cS V$. Here $V$ is a free non-negatively graded $\mathcal{D}$-module, which admits a homogeneous basis $(v_\za)_{\za\in J}$ that is indexed by a well-ordered set $J$, and is such that \be\label{Lowering}d v_\za \in A\boxtimes \cS V_{<\za}\;,\ee for all $\za\in J$. In the last requirement, we set $V_{<\za}:=\bigoplus_{\zb<\za}\cD\cdot v_\zb\,$. We refer to Property (\ref{Lowering}) by saying that $d$ is {\bf lowering}.

A {\small RS$\cD\!$A} with Property (\ref{minimal}) $(\,$resp., with Property (\ref{split}); over $(A,d_A)=(\cO,0)$$\,)$ is called a {\bf minimal} {\small RS$\cD\!$A} $(\,$resp., a {\bf split} {\small RS$\cD\!$A}; a {\bf Sullivan $\cD$-algebra} $(\,${\small S$\cD\!$A}$\,)$$\,)$  and it is often simply denoted by $(A\boxtimes \cS V,d)$ $(\,$resp., $(A\otimes \cS V,d)$; $(\cS V,d)$$\,)$.
\end{defi}

The next two lemmas are of interest for the split situation.

\begin{lem}\label{DiffGen} Let $(v_{\za})_{\za\in I}$ be a family of generators of homogeneous non-negative degrees, and let $$V:=\langle v_\za: \za\in I\ra:=\bigoplus_{\za\in I}\,\cD\cdot v_\za$$ be the free non-negatively graded $\cD$-module over $(v_\za)_{\za\in I}$. Then, any degree $-1$ map $d\in {\tt Set}((v_\za),V)$ uniquely extends to a degree $-1$ map $d\in{\tt \cD M}(V,V)$. If moreover $d^2=0$ on $(v_\za)$, then $(V,d)\in\tt DG\cD M\,.$  \end{lem}

Since $\cS V$ is the free differential graded $\cD$-algebra over the differential graded $\cD$-module $V$, a morphism $f\in {\tt DG\cD A}(\cS V,B),$ valued in $(B,d_B)\in {\tt DG\cD A}$, is completely defined by its restriction $f\in {\tt DG\cD M}(V,B)$. Hence, the

\begin{lem}\label{MorpGen} Consider the situation of Lemma \ref{DiffGen}. Any degree 0 map $f\in {\tt Set}((v_\za), B)$ uniquely extends to a morphism $f\in{\tt G\cD M}(V,B)$. Furthermore, if $d_B\,f=f\,d$ on $(v_\za)$, this extension is a morphism $f\in{\tt DG\cD M}(V,B),$ which in turn admits a unique extension $f\in{\tt DG\cD A}(\cS V,B)$.\end{lem}

\subsection{Quillen's transfer theorem}

We use the adjoint pair \be\label{Ad}{\cal S}:{\tt DG\cD M}\rightleftarrows{\tt DG\cD A}:\op{For}\;\ee to transfer the cofibrantly generated model structure from the source category $\tt DG\cD M$ to the target category $\tt DG\cD A$. This is possible if Quillen's transfer theorem \cite{GS} applies.

\begin{theo}\label{QTT}
Let $ F : {\tt C} \rightleftarrows {\tt D} : G $ be a pair of adjoint functors. Assume that $\tt C$ is a cofibrantly generated model category and denote by $I$ (resp., $J$) its set of generating cofibrations (resp., trivial cofibrations). Define a morphism $f : X \to Y$ in $\tt D$ to be a weak equivalence (resp., a fibration), if $Gf$ is a weak equivalence (resp., a fibration) in $\tt C$. If
\begin{enumerate}
\item
the right adjoint $G : {\tt D} \to {\tt C}$ commutes with sequential colimits, and
\item
any cofibration in $\tt D$ with the {\small LLP} with respect to all fibrations is a weak equivalence,
\end{enumerate}
then $\tt D$ is a cofibrantly generated model category that admits $\{ Fi: i \in I \}$ (resp., $\{ Fj : j \in J \}$) as set of generating cofibrations (resp., trivial cofibrations).
\end{theo}

Of course, in this version of the transfer principle, the mentioned model structures are cofibrantly generated model structures in the sense of \cite{GS}.\medskip

Condition 2 is the main requirement of the transfer theorem. It can be checked using the following lemma \cite{GS}:

\begin{lem}\label{SuffCondFor2}
Assume in a category $\tt D$ (which is not yet a model category, but has weak equivalences, and fibrations),
\begin{enumerate}
\item there is a functorial fibrant replacement functor, and
\item every object has a natural path object, i.e., for any $D\in \tt D$, we have a natural commutative
diagram

\begin{center}
\begin{tikzpicture}
\node(C){$D$};
\node(A)[right of=C, xshift=2cm]{$D\times D$};
\node(P)[above of=A,yshift=2cm]{$\op{Path}(D)$};
\draw[->](C) to node[above]{${{\Delta}}$} (A);
\draw[->](C) to node{$\hspace{-30pt}i$} (P);
\draw[->](P) to node[below]{$\hspace{17pt} q$} (A);
\end{tikzpicture}
\end{center}

\end{enumerate}
where $\zD$ is the diagonal map, $i$ is a weak equivalence and $q$ is a fibration.
Then every cofibration in $\tt D$ with the {\small LLP} with respect to all fibrations is a weak
equivalence.
\end{lem}

We think about $\op{Path}(D)\in\tt D$ is an internalized `space' of paths in $D$. In simple cases, $\op{Path}(D)=\h_{\tt D}(I,D)$, where $I\in\tt D$ and where $\h_{\tt D}$ is an internal Hom. Moreover, by fibrant replacement of an object $D\in\tt D$, we mean a weq $D\to \bar{D}$ whose target is a fibrant object.

\subsection{Proof of Condition 1 of Theorem \ref{QTT}}

Let $\zl$ be a non-zero ordinal and let $X:\zl\to \tt C$ be a diagram of type $\zl$ in a category $\tt C$, i.e., a functor from $\zl$ to $\tt C$. Since an ordinal number is a totally ordered set, the considered ordinal $\zl$ can be viewed as a directed poset $(\zl,\le)$. Moreover, the diagram $X$ is a direct system in $\tt C$ over $\zl$ -- made of the $\tt C$-objects $X_\zb$, $\zb<\zl$, and the $\tt C$-morphisms $X_{\zb\zg}:X_\zb\to X_\zg$, $\zb\le\zg$ -- , and the colimit $\colim_{\zb<\zl}X_\zb$ of this diagram $X$ is the inductive limit of the system $(X_\zb,X_{\zb\zg})$.\medskip

Let now $A:\zl\to \tt DG\cD A$ be a diagram of type $\zl$ in $\tt DG\cD A$ and let $\op{For}\circ A:\zl\to \tt DG\cD M$ be the corresponding diagram in $\tt DG\cD M$. If no confusion arises, we denote the latter diagram simply by $A$. As mentioned in the proof of Theorem \ref{FinGenModDGDM}, the colimit of $A$ does exist in $\tt DG\cD M$ and is taken degree-wise in ${\tt Mod}(\cD)$. More precisely, for all $r\in \N$, we denote by $\op{Pr}_r$ the canonical functor ${\tt DG\cD M}\to {\tt Mod}(\cD)$ and consider the diagram $A_r:=\op{Pr}_r\circ A:\zl\to {\tt Mod}(\cD)$. In view of the preceding paragraph, the colimit $\colim_{\zb<\zl}A_{\zb,r}$ of $A_r$ in ${\tt Mod}(\cD)$ is the inductive limit in ${\tt Mod}(\cD)$ of the direct system $(A_{\zb,r}, A_{\zb\zg,r})$. A well-known construction provides the inductive limit $C_r$ of this direct system in $\tt Set$: $$C_r=\coprod_{\zb<\zl} A_{\zb,r}/\sim\;,$$ where $a_{\zb,r}\sim a'_{\zg,r}\,$, if there is $\zd$ such that $A_{\zb\zd,r}\;a_{\zb,r}=A_{\zg\zd,r}\;a'_{\zg,r}$. The set $C_r$ can be made an object $C_r\in{\tt Mod}(\cD)$ in a way such that the projections $\zp_{\zb,r}:A_{\zb,r}\to C_r$ become ${\tt Mod}(\cD)$-morphisms and $C_r$ becomes the colimit in ${\tt Mod}(\cD)$ of $A_r$ -- in particular $\zp_{\zg,r}\,A_{\zb\zg,r}=\zp_{\zb,r}$. Due to universality, there is a ${\tt Mod}(\cD)$-morphism $d_{r}:C_r\to C_{r-1}$ such that $d_r\zp_{\zb,r}=\zp_{\zb,r-1}d_{\zb,r}$, where $d_{\zb,r}$ is the differential $d_{\zb,r}:A_{\zb,r}\to A_{\zb,r-1}$. We thus get a complex $(C_\bullet,d)\in \tt DG\cD M$, together with $\tt DG\cD M$-morphisms $\zp_{\zb,\bullet}:A_{\zb,\bullet}\to C_\bullet$, and this complex is the colimit in $\tt DG\cD M$ of $A$.

We now define on $C_\bullet$ a multiplication $\diamond$ in the usual way: if $c_r\in C_r$ and $c_s\in C_s$, $$c_r\diamond c_s = \zp_{\zb,r}\,a_{\zb,r}\diamond \zp_{\zg,s}\,a'_{\zg,s}=\zp_{\zd,r}\,A_{\zb\zd,r}\,a_{\zb,r}\diamond \zp_{\zd,s}\,A_{\zg\zd,s}\,a'_{\zg,s}$$ $$=\zp_{\zd,r+s}\,\left(A_{\zb\zd,r}\,a_{\zb,r}\star A_{\zg\zd,s}\,a'_{\zg,s}\right)\in C_{r+s}\;,$$ where $\star$ denotes the multiplication of an element in $A_{\zd,r}$ and an element in $A_{\zd, s}$, in the {\small DG} $\cal D$-algebra $A_{\zd,\bullet}\,$. It is straightforwardly checked that $\diamond$ is a well-defined graded commutative unital $\cO$-algebra structure on $C_\bullet\,$, such that the differential of $C_\bullet$ is a degree $-1$ graded derivation of $\diamond$ and that vector fields act as non-graded derivations. Hence, $(C_\bullet,d,\diamond)$ is an object $C_\bullet\in \tt DG\cD A$ and the maps $\zp_{\zb,\bullet}:A_{\zb,\bullet}\to C_\bullet$ are $\tt DG\cD A$-morphisms. It is now easily seen that $C_\bullet$ is the colimit in $\tt DG\cD A$ of $A$.\medskip

Hence, the

\begin{prop}\label{Enrichment} For any ordinal $\zl$, the colimit in $\tt DG\cD A$ of any diagram $A$ of type $\zl$ exists and, if $\zl$ is non-zero, it is obtained as an enrichment of the corresponding colimit in $\tt DG\cD M$: \be\label{EnrichmentEq}\op{For}(\op{colim}_{\zb<\zl}A_{\zb,\bullet})=\op{colim}_{\zb<\zl}\op{For}(A_{\zb,\bullet})\;.\ee\end{prop}

If $\zl$ is the zero ordinal, it can be viewed as the empty category $\emptyset$. Therefore, the colimit in $\tt DG\cD A$ of the diagram of type $\zl$ is in this case the initial object $(\cO,0)$ of $\tt DG\cD A$. Since the initial object in $\tt DG\cD M$ is $(\{0\},0)$, we see that $\op{For}$ does not commute with this colimit. The above proof fails indeed, as $\emptyset$ is not a directed set.\medskip

It follows from Proposition \ref{Enrichment} that the right adjoint $\op{For}$ in (\ref{Ad}) commutes with sequential colimits, so that the first condition of Theorem \ref{QTT} is satisfied.

\subsection{Proof of Condition 2 of Theorem \ref{QTT}}\label{Condition2}

We prove Condition 2 using Lemma \ref{SuffCondFor2}. In our case, the adjoint pair is $${\cal S}:{\tt DG\cD M}\rightleftarrows{\tt DG\cD A}:\op{For}\;.$$ As announced in Subsection \ref{RSDA}, we omit $\bullet$, $\star$, and $\cO$, whenever possible. It is clear that every object $A\in{\tt D}={\tt DG\cD A}$ is fibrant. Hence, we can choose the identity as fibrant replacement functor, with the result that the latter is functorial.\medskip

As for the second condition of the lemma, we will show that {\it any} $\,\tt DG\cD A$-morphism $\zf:A\to B$ naturally factors into a weak equivalence followed by a fibration.\medskip

Since in the standard model structure on the category of differential graded commutative algebras over $\Q$, cofibrations are retracts of relative Sullivan algebras \cite{Hes}, the obvious idea is to decompose $\zf$ as $A\to A\0\cS V\to B$, where $i: A\to A\0\cS V$ is a (split minimal) relative Sullivan $\cD$-algebra, such that there is a projection $p: A\0\cS V\to B$, or, even better, a projection $\ze: V\to B$ in positive degrees. The first attempt might then be to use $$\ze:V=\bigoplus_{n>0}\bigoplus_{b_n\in B_n}\cD\cdot 1_{b_n}\ni 1_{b_n}\mapsto b_n\in B\;,$$ whose source incorporates a copy of the sphere $S^n$ for each $b_n\in B_n$, $n>0\,.$ However, $\ze$ is not a chain map, since in this case we would have $d_Bb_n=d_B \ze 1_{b_n}=0$, for all $b_n$. The next candidate is obtained by replacing $S^n$ by $D^n$: if $B\in {\tt DG\cD M}$, set $$P(B)=\bigoplus_{n>0}\bigoplus_{b_n\in B_n}D^n_{\bullet}\in {\tt DG\cD M}\;,$$ where $D^n_{\bullet}$ is a copy of the $n$-disc $$D^n_{\bullet}: \cdots \to 0\to 0\to \cD\cdot \mathbb{I}_{b_n} \to \cD\cdot s^{-1}\mathbb{I}_{b_n}\to 0\to \cdots\to 0\;.$$ Since $$P_n(B)=\bigoplus_{b_{n+1}\in B_{n+1}}\cD\cdot s^{-1}\mathbb{I}_{b_{n+1}}\oplus \bigoplus_{b_n\in B_n}\cD\cdot \mathbb{I}_{b_n}\;\; (n>0)\quad\text{and}\quad P_0(B)=\bigoplus_{b_1\in B_1}\cD\cdot s^{-1}\mathbb{I}_{b_1}\;,$$ the free non-negatively graded $\cD$-module $P(B)$ is projective in each degree, what justifies the chosen notation. On the other hand, the differential $d_P$ of $P(B)$ is the degree $-1$ square 0 $\cD$-linear map induced by the differentials in the $n$-discs and thus defined on $P_n(B)$ by $$d_P(s^{-1}\mathbb{I}_{b_{n+1}})=0\in P_{n-1}(B)\quad\text{and}\quad d_P(\mathbb{I}_{b_n})=s^{-1}\mathbb{I}_{b_n}\in P_{n-1}(B)\;$$ (see Lemma \ref{DiffGen}). The canonical projection $\ze:P(B)\to B\,$, is defined on $P_n(B)$, as degree 0 $\cD$-linear map, by $$\ze(s^{-1}\mathbb{I}_{b_{n+1}})=d_B(b_{n+1})\in B_n\quad\text{and}\quad\ze(\mathbb{I}_{b_n})=b_n\in B_n\;.$$ It is clearly a $\tt DG\cD M$-morphism and extends to a $\tt DG\cD A$-morphism $\ze:\cS(P(B))\to B$ (see Lemma \ref{MorpGen}).\medskip

We define now the aforementioned $\tt DG\cD A$-morphisms $i:A\to A\0 \cS(P(B))$ and $p:A\0 \cS(P(B))\to B$, where $i$ is a weak equivalence and $p$ a fibration such that $p\circ i=\zf\,.$ We set $i=\id_A\0 1$ and $p=\zm_B\circ (\zf\0 \ze)\,.$ It is readily checked that $i$ and $p$ are $\tt DG\cD A$-morphisms (see Proposition \ref{ProdTensProd}) with composite $p\circ i=\zf\,.$ Moreover, by definition, $p$ is a fibration in $\tt DG\cD A$, if it is surjective in degrees $n>0$ -- what immediately follows from the fact that $\ze$ is surjective in these degrees.\medskip

It thus suffices to show that $i$ is a weak equivalence in $\tt DG\cD A$, i.e., that $$H(i):H(A)\ni [a]\to [a\0 1]\in H\left(A\0{\cal S}(P(B))\right)$$ is an isomorphism of graded $\cD$-modules. Since $\tilde \imath:A\to A\0\cO$ is an isomorphism in $\tt DG\cD M$, it induces an isomorphism $$H(\tilde \imath):H(A)\ni[a]\to [a\0 1]\in H(A\0\cO)\;.$$ In view of the graded $\cD$-module isomorphism $$H(A\0{\cal S}(P(B)))\simeq H( A\0\cO)\oplus H(A\0{\cal S}^{\ast\ge 1}(P(B)))\;,$$ we just have to prove that \be H(A\0{\cal S}^{k\ge 1}(P(B)))=0\;\label{SCond1}\ee as graded $\cD$-module, or, equivalently, as graded $\cO$-module.\medskip

To that end, note that $$0\longrightarrow \ker^{k}{\frak S}\stackrel{\iota}{\longrightarrow}P(B)^{\0 k}\stackrel{\frak S}{\longrightarrow}(P(B)^{\0 k})^{\mathbb{S}_k}\longrightarrow 0\;,$$ where $k\ge 1$ and where $\frak S$ is the averaging map, is a short exact sequence in the Abelian category $\tt DG\cO M$ of differential non-negatively graded $\cO$-modules (see Appendix \ref{InvCoinv}, in particular Equation (\ref{SymOp})). Since it is canonically split by the injection $${\frak I}:(P(B)^{\0 k})^{\mathbb{S}_k}\to P(B)^{\0 k}\;,$$ and $$(P(B)^{\0 k})^{\mathbb{S}_k}\simeq {\cal S}^{k}(P(B))$$ as {\small DG} $\cO$-modules (see Equation (\ref{Alg2})), we get $$P(B)^{\0 k}\simeq \cS^k(P(B))\oplus \ker^k{\frak S}\quad\text{and}\quad A\0 P(B)^{\0 k}\simeq A\0\cS^k(P(B))\,\oplus\, A\0\ker ^k{\frak S}\;,$$ as {\small DG} $\cO$-modules. Therefore, it suffices to show that the {\small LHS} is an acyclic chain complex of $\cO$-modules.\medskip

We begin showing that $\cD=\cD_X(X)$, where $X$ is a smooth affine algebraic variety, is a flat module over $\cO=\cO_X(X)$. Note first that, the equivalence (\ref{ShVsSectqcMod1}) $$\zG(X,\bullet):{\tt qcMod}(\cO_X)\rightleftarrows {\tt Mod}(\cO):\widetilde{\bullet}$$ is exact and strongly monoidal (see remark below Equation (\ref{ShVsSectqcMod1})). Second, observe that $\cD_X$ is a locally free $\cO_X$-module, hence, a flat (and quasi-coherent) sheaf of $\cO_X$-modules, i.e., $\cD_X\0_{\cO_X}\bullet\,$ is exact in ${\tt Mod}(\cO_X)$. To show that $\cD\0_\cO\bullet$ is exact in ${\tt Mod}(\cO)$, consider an exact sequence $$0\to M'\to M\to M''\to 0$$ in ${\tt Mod}(\cO)$. From what has been said it follows that $$0\to \cD_X\0_{\cO_X}\widetilde{M'}\to \cD_X\0_{\cO_X}\widetilde{M}\to \cD_X\0_{\cO_X}\widetilde{M''}\to 0$$ is an exact sequence in ${\tt Mod}(\cO_X)$, as well as an exact sequence in ${\tt qcMod}(\cO_X)$ (kernels and cokernels of morphisms of quasi-coherent modules are known to be quasi-coherent). When applying the exact and strongly monoidal global section functor, we see that $$0\to\cD\0_\cO M'\to \cD\0_\cO M\to \cD\0_\cO M''\to 0$$ is exact in ${\tt Mod}(\cO)$.\medskip

Next, observe that
$$
H(A\0 P(B)^{\0 k})=\bigoplus_{n>0}\bigoplus _{b_n\in B_n} H(D^n_\bullet\0 A \otimes P(B)^{\otimes (k-1)})\;.
$$
To prove that each of the summands of the {\small RHS} vanishes, we apply K\"unneth's Theorem \cite[Theorem 3.6.3]{Wei93} to the complexes $D_\bullet^n$ and $A \otimes P(B)^{\otimes (k-1)}$, noticing that both, $D^n_\bullet$ (which vanishes, except in degrees $n,n-1$, where it coincides with $\cD$) and $d(D^n_\bullet)$ (which vanishes, except in degree $n-1$, where it coincides with $\cD$), are termwise flat $\mathcal{O}$-modules. We thus get, for any $m$, a short exact sequence
\begin{align}\nonumber
0\rightarrow \bigoplus_{p+q=m}H_p(D^n_\bullet)\otimes H_q(A \otimes P(B)^{\otimes (k-1)})&\rightarrow H_m(D_\bullet^n\0 A \otimes P(B)^{\otimes (k-1)})\rightarrow\\ \nonumber
&\bigoplus_{p+q=m-1} \op{Tor}_1(H_p(D_\bullet^n), H_q(A \otimes P(B)^{\otimes (k-1)}))\rightarrow 0\;.
\end{align}
Finally, since $D^n_\bullet$ is acyclic, the central term of this exact sequence vanishes, since both, the first and the third, do.\medskip

To completely finish checking the requirements of Lemma \ref{SuffCondFor2} and thus of Theorem \ref{QTT}, we still have to prove that the factorization $(i,p)=(i(\zf),p(\zf))$ of $\zf$ is functorial. In other words, we must show that, for any commutative $\tt DG\cD A$-square \be\label{InitSq}
\xymatrix{
A\ar[d]^{u} \ar[r]^{\zf}&B\ar[d]^{v\;\;\;,}\\
A'\ar[r]^{\zf'}&B'\\
}
\ee
there is a commutative $\tt DG\cD A$-diagram
\be\label{CompMor00}
\xymatrix{A\;\; \ar[d]_{u} \ar^{\sim}_{i(\zf)}  @{->} [r] & A\0\cS U \ar[d]^{w}\;\;\ar @{->>} [r]_{p(\zf)} & B \ar[d]^{v\;\;\;,}\\
A'\;\; \ar @{->} [r]^{\sim}_{i(\zf')} & A'\0\cS U'\;\; \ar @{->>} [r]_{p(\zf')} & B'\\
}
\ee
where we wrote $U$ (resp., $U'$) instead of $P(B)$ (resp., $P(B')$).

To construct the $\tt DG\cD A$-morphism $w$, we first define a $\tt DG\cD A$-morphism $\tilde{v}:\cS U\to \cS U'$, then we obtain the $\tt DG\cD A$-morphism $w$ by setting $w=u\0 \tilde{v}$.

To get the $\tt DG\cD A$-morphism $\tilde{v}$, it suffices, in view of Lemma \ref{MorpGen}, to define a degree 0 $\tt Set$-map $\tilde{v}$ on $G:=\{s^{-1}\mbi_{b_n},\mbi_{b_n}:b_n\in B_n,n>0\}$, with values in the differential graded $\cD$-algebra $(\cS U',d_{U'})$, which satisfies $d_{U'}\,\tilde{v}=\tilde{v}\,d_U$ on $G$. We set $$\tilde{v}(s^{-1}\mbi_{b_n})=s^{-1}\mbi_{v(b_n)}\in\cS U'\;\,\text{and}\;\,\tilde{v}(\mbi_{b_n})=\mbi_{v(b_n)}\in\cS U'\;,$$ and easily see that all the required properties hold.

We still have to verify that the diagram (\ref{CompMor00}) actually commutes. Commutativity of the left square is obvious. As for the right square, let $t:={a}\0 x_1\odot\ldots\odot x_k\in A\0\cS U$, where the $x_i$ are elements of $U$, and note that $$v\, p(\zf)(t)= v\, (\zm_B\circ (\zf\0 \ze))(t)=v\,\zf({a})\star v\,\ze(x_1)\star\ldots\star v\,\ze(x_k)$$ and $$p(\zf')w(t)=(\zm_{B'}\circ(\zf'\0\ze'))(u({a})\0 \tilde{v}(x_1)\odot\ldots\odot \tilde{v}(x_k))$$ $$=\zf'u({a})\star\,\ze'\, \tilde{v}(x_1)\,\star\,\ldots\,\star\, \ze'\, \tilde{v}(x_k)\;,$$ where $\star$ denotes the multiplication in $B'$. Since the square (\ref{InitSq}) commutes, it suffices to check that \be\label{ComRel}v\,\ze(x)=\ze'\,\tilde{v}(x)\;,\ee for any $x\in U\,.$ However, the $\cD$-module $U$ is freely generated by $G$ and the four involved morphisms are $\cD$-linear: it is enough that (\ref{ComRel}) holds on $G$ -- what is actually the case.

\subsection{Transferred model structure}

We proved in Theorem \ref{FinGenModDGDM} that $\tt DG\cD M$ is a finitely generated model category whose set of generating cofibrations (resp., trivial cofibrations) is \be\label{GenCof1} I=\{\iota_k: S^{k-1}_\bullet \to D^k_\bullet, k\ge 0\}\ee $(\,$resp., \be\label{GenTrivCof1}J=\{\zeta_k: 0 \to D^k_\bullet, k\ge 1\}\;)\;.\ee Theorem \ref{QTT} thus allows to conclude that:

\begin{theo}\label{FinGenModDGDA} The category $\tt DG\mathcal{D}A$ of differential non-negatively graded commutative $\cD$-algebras is a finitely $(\,$and thus a cofibrantly$\,)$ generated model category $(\,$in the sense of \cite{GS} and in the sense of \cite{Hov}$\,)$, with $\cS I=\{\cS \iota_k:\iota_k\in I\}$ as its generating set of cofibrations and $\cS J=\{\cS \zeta_k: \zeta_k\in J\}$ as its generating set of trivial cofibrations. The weak equivalences are the $\tt DG\cD A$-morphisms that induce an isomorphism in homology. The fibrations are the $\tt DG\cD A$-morphisms that are surjective in all positive degrees $p>0$.\end{theo}

The cofibrations will be described below.\medskip

Quillen's transfer principle actually provides a \cite{GS} cofibrantly generated (hence, a \cite{Hov} cofibrantly generated) \cite{GS} model structure on $\tt DG\cD A$ (hence, a \cite{Hov} model structure, if we choose for instance the functorial factorizations given by the small object argument). In fact, this model structure is finitely generated, i.e. (see Appendix \ref{CofGenModCat}), the domains and codomains of the maps in $\cS I$ and $\cS J$ are $n$-small $\tt DG\cD A$-objects, $n\in\N$, relative to $\op{Cof}$. Indeed, these sources and targets are $\cS D^k_\bullet$ ($k\ge 1$), $\cS S^k_\bullet$ ($k\ge 0$), and $\cO$. We already observed (see Theorem \ref{FinGenModDGDM}) that $D^k_\bullet$ ($k\ge 1$), $S^k_\bullet$ ($k\ge 0$), and $0$ are $n$-small $\tt DG\cD M$-objects with respect to all $\tt DG\cD M$-morphisms. If $\frak S_\bullet$ denotes any of the latter chain complexes, this means that the covariant Hom functor $\op{Hom}_{\tt DG\cD M}({\frak S}_\bullet,-)$ commutes with all $\tt DG\cD M$-colimits $\colim_{\zb<\zl}M_{\zb,\bullet}$ for all limit ordinals $\zl$. It therefore follows from the adjointness property (\ref{Adjoint}) and the enrichment equation (\ref{EnrichmentEq}) that, for any $\tt DG\cD A$-colimit $\colim_{\zb<\zl}A_{\zb,\bullet}$, we have $$\h_{\tt DG\cD A}(\cS {\frak S}_\bullet,\colim_{\zb<\zl}A_{\zb,\bullet})\simeq \h_{\tt DG\cD M}({\frak S}_\bullet,\op{For}(\colim_{\zb<\zl}A_{\zb,\bullet}))=$$ $$\h_{\tt DG\cD M}({\frak S}_\bullet,\colim_{\zb<\zl}\op{For}(A_{\zb,\bullet}))=\colim_{\zb<\zl}\h_{\tt DG\cD M}({\frak S}_\bullet,\op{For}(A_{\zb,\bullet}))\simeq$$ $$\colim_{\zb<\zl}\h_{\tt DG\cD A}(\cS {\frak S}_\bullet,A_{\zb,\bullet})\;.$$

\subsection{First insight into cofibrations}

The main idea in the above verification of the requirements of the transfer theorem is the decomposition of an arbitrary $\tt DG\cD A$-morphism $\zf:A\to B$ into a trivial `cofibration' $i:A\to A\0\cS U$ and a fibration $p:A\0\cS U\to B$. Indeed, it is implicit in Subsection \ref{Condition2} that the $\tt DG\cD A$-cofibrations are exactly the retracts of the relative Sullivan $\cD$-algebras and that $i$ is a split minimal relative Sullivan $\cD$-algebra. The former result is proven in \cite{BPP2}. The latter is almost obvious. Indeed, \be\label{U}U=P(B)=\bigoplus_{n>0}\bigoplus_{b_n\in B_n}D^n_\bullet\in\tt DG\cD M\ee with differential $d_U=d_P$ defined by \be\label{dU}d_U(s^{-1}\mathbb{I}_{b_n})=0\quad\text{and}\quad d_U(\mathbb{I}_{b_n})=s^{-1}\mathbb{I}_{b_n}\;.\ee Hence, $\cS U\in\tt DG\cD A$, with differential $d_S$ induced by $d_U$, and $A\0 \cS U\in\tt DG\cD A$, with differential \be\label{d}d_1=d_A\0 \id +\id\0 d_S\;.\ee Therefore, $i:A\to A\0\cS U$ is a $\tt DG\cD A$-morphism. Since $U$ is the free non-negatively graded $\cD$-module with homogeneous basis $$G=\{s^{-1}\mathbb{I}_{b_n}, \mathbb{I}_{b_n}:b_n\in B_n, n>0\}\;,$$ all the requirements of the definition of a split minimal {\small RS$\cD$A} are obviously satisfied, except that we still have to check the well-ordering, the lowering, and the minimality conditions.\medskip

Since every set can be well-ordered, we first choose a well-ordering on each $B_n$, $n>0$: if $\zl_n$ denotes the unique ordinal that belongs to the same $\cI$-equivalence class (see Appendix \ref{Small}), the elements of $B_n$ can be viewed as labelled by the elements of $\zl_n$. Then we define the following total order: the $s^{-1}\mathbb{I}_{b_1}$, $b_1\in B_1$, are smaller than the $\mathbb{I}_{b_1}$, which are smaller than the $s^{-1}\mathbb{I}_{b_2}$, and so on ad infinitum. The construction of an infinite decreasing sequence in this totally ordered set amounts to extracting an infinite decreasing sequence from a finite number of ordinals $\zl_1,\zl_1,\ldots,\zl_k$. Since this is impossible, the considered total order is a well-ordering. The lowering condition is thus a direct consequence of Equations (\ref{dU}) and (\ref{d}).\medskip

Let now $\{\zg_\za:\za \in J\}$ be the set $G$ of generators endowed with the just defined well-order. Observe that, if the label $\za$ of the generator $\zg_\za$ increases, its degree $\deg\zg_\za$ increases as well, i.e., that \be\label{Mini}\za\le\zb\quad \Rightarrow\quad \deg \zg_\za\le\deg\zg_\zb\;.\ee Eventually, as announced:

\begin{theo}\label{TrivCof-Fib} Any $\tt DG\cD A$-morphism $\zf:A\to B$ can be functorially decomposed into a weak equivalence $i:A\to A\0\cS U$ (see (\ref{U})), which is a split minimal {\small RS$\cD$A}, and a fibration $p:A\0\cS U\to B$.\end{theo}

For the above-mentioned result on $\tt DG\cD A$-cofibrations, an explicit description of fibrant and cofibrant functorial replacement functors in $\tt DG\cD A$, as well as a model categorical Koszul-Tate resolution, we refer the reader to \cite{BPP2}.

\section{Appendices}

The following appendices do not contain new results but might have a pedagogical value. Various (also online) sources were used. Notation is the same as in the main part of the text.

\subsection{Appendix 1 -- Coherent and quasi-coherent sheaves of modules}\label{FinCondShMod}

A {\bf finitely generated $\cR$-module} is an object $\cP\in{\tt Mod}(\cR)$ that (as a sheaf of modules) is locally generated by a finite number of sections. More precisely, for any $x\in X$, there exists a neighborhood $U\ni x$, and {\it a finite number of sections} $s_1,\ldots,s_n\in\cP(U)$, such that the sheaf morphism $\zf:\cR^n|_U\to \cP|_U$, which is defined, for any $V\in{\tt Open}_U$, by $\zf_V:\cR(V)^n\ni (f^1,\ldots,f^n)\mapsto \sum_if^is_i|_V\in \cP(V)$, is epic. In other words, for any $x\in X$, there is a neighborhood $U\ni x$, an integer $n\in\N$, and an exact sequence of sheaves \be\label{FinGen}\cR^n|_U\to {\cal P}|_U\to 0\;.\ee Indeed, when choosing the definition (\ref{FinGen}), we recover the generating sections as follows. If $e_i=(0,\ldots,0,1,0,\ldots,0)\in\cR(U)^n$, we can set $s_i=\zf_U(e_i)\in\cP(U)$. Then, any tuple $(f^1,\ldots,f^n)\in\cR(V)^n$ reads $\sum_if^ie_i|_V$ and is mapped by $\zf_V$ to $\sum_if^is_i|_V\,.$\medskip

A {\bf finitely presented $\cR$-module} is an object $\cP\in{\tt Mod}(\cR)$ that is locally generated by a finite number of sections, which satisfy a finite number of relations. This means that, for any $x\in X$, there is a neighborhood $U\ni x$, and a finite number of sections $s_1,\ldots,s_n\in\cP(U)$, such that the induced morphism $\zf:\cR^n|_U\to \cP|_U$ is a sheaf epimorphism, and such that the sheaf $$\ker\zf=\ker\lp \cR^n|_U\ni (f^1,\ldots,f^n)\mapsto \sum_if^is_i\in \cP|_U\rp$$ of relations is finitely generated. In other words, for every $x\in X$, there is $U\ni x$, integers $n,m\in\N$, and an exact sequence of sheaves \be\label{FinPre}\cR^{m}|_U\to{\cR}^n|_U\to {\cP}|_U\to 0\;.\ee Of course, finitely presented implies finitely generated.

A {\bf coherent $\cR$-module} is an object $\cP\in{\tt Mod}(\cR)$ that is finitely generated and, for any open subset $U$ and any sections $t_1,\ldots,t_k\in\cP(U)$, the relation sheaf $$\ker\lp\cR^k|_U\ni (f^1,\ldots,f^k)\mapsto \sum_if^it_i\in \cP|_U\rp$$ is finitely generated. This means that $\cP$ is finitely generated and, for any open subset $U$, any $k\in\N$, and any sheaf morphism $\psi:\cR^k|_U\to\cP|_U$, the kernel sheaf \be\label{Coh}\ker\psi=\ker\lp\cR^k|_U\to\cP|_U\rp\ee is finitely generated. The category ${\tt cMod}(\cR)$ of coherent $\cR$-modules is reasonably behaved. It is closed under usual operations such as kernels, cokernels, finite direct sums... Actually, it is a full Abelian subcategory of the Abelian category ${\tt Mod}(\cR)$ (an Abelian subcategory $\tt S$ of an Abelian category $\tt C$ is a subcategory that is Abelian and such that any exact sequence in $\tt S$ is also exact in $\tt C$). Moreover, coherent always implies finitely presented, and, if $\cR$ is coherent (as module over itself), an arbitrary $\cR$-module is coherent if and only if it is finitely presented.\medskip

A {\bf quasi-coherent $\cR$-module} is an object $\cP\in{\tt Mod}(\cR)$ that is locally presented, i.e., for any $x\in X$, there is a neighborhood $U\ni x$, such that there is an exact sequence of sheaves
\be\label{QuaCoh}\cR^{K_U}|_{U}\to \cR^{J_U}|_{U}\to \cP|_{U}\to 0\;,\ee where $\cR^{K_U}$ and $\cR^{J_U}$ are (not necessarily finite) direct sums. Let us recall that an infinite direct sum of sheaves need not be a sheaf, so that a sheafification is required. The category ${\tt qcMod}(\cR)$ of quasi-coherent $\cR$-modules is not Abelian in general, but is Abelian in the context of Algebraic Geometry, i.e., if $\cR$ is the function sheaf of a scheme.\medskip

A {\bf locally free (resp., locally finite free, locally free of finite rank $r$) $\cR$-module} is an object $\cP\in{\tt Mod}(\cR)$, such that, for any $x\in X$, there is an open neighborhood $U\ni x$, a set $I$ (resp., a finite set $I$, a finite set $I$ of cardinality $r$), and an exact sequence of sheaves \be\label{LocFree}0\to \cR^I|_U\to\cP|_U\to 0\;.\ee It is clear that locally free implies quasi-coherent. However, any locally finite free $\cR$-module is coherent if and only if $\cR$ is coherent.\medskip


The category of locally free modules is not Abelian, basically because cokernels are not locally free. However, if $\cR$ is a function sheaf $\cO_X$, locally free modules ${\tt lfMod}(\cO_X)$ (resp., locally free modules ${\tt lfrMod}(\cO_X)$ of finite rank) sit in broader Abelian categories:

\be\label{SeqAbelCatAlgGeo} {\tt lfMod}(\cO_X)\subset {\tt qcMod}(\cO_X)\subset {\tt Mod}(\cO_X)\;,\ee if $X$ is a scheme, and \be {\tt lfrMod}(\cO_X)\subset {\tt cMod}(\cO_X)\subset {\tt qcMod}(\cO_X)\;,\ee if $X$ is a Noetherian scheme (i.e., a scheme that is finitely covered by spectra of Noetherian rings). 

\subsection{Appendix 2 -- $\cD$-modules}\label{D-modules}

We already indicated that $\cD$-modules are fundamental in algebraic analysis: they allow to apply methods of homological algebra and sheaf theory to the study of systems of {\small PDE}-s \cite{KS}.\medskip

We first explain the key idea of Proposition \ref{DModFlatConnSh} considering -- to simplify -- total sections instead of sheaves.\medskip

We denote by $\cD$ the ring of differential operators acting on functions of a suitable base space $X$, e.g., a finite-dimensional smooth manifold \cite{Cos}. A {$\cD$-module} $M\in{\tt Mod}(\cD)$ (resp., $M\in{\tt Mod}(\cD^{\op{op}})$) is a left (resp., right) module over the noncommutative ring $\cD$. Since {\it $\cD$ is generated by smooth functions $f\in\cO$ and smooth vector fields $\theta\in\Theta$, modulo the obvious commutation relations between these types of generators, a $\cD$-action on an $\cO$-module $M\in{\tt Mod}(\cO)$ is completely defined if it is given for vector fields, modulo the commutation relations}. More precisely, let $$\cdot\,:\cO\times M\ni(f,m)\mapsto f\cdot m\in M$$ be the $\cO$-action, and let \be\label{FC1}\nabla:\Theta\times M\ni (\theta,m)\mapsto \nabla_\theta m\in M\ee be an $\R$-bilinear `$\Theta$-action'. For $f\in\cO$ and $\theta,\theta'\in\Theta$, we then naturally extend $\nabla$ by defining the action $\nabla_{\theta\theta'}$ (resp., $\nabla_{\theta f}$) of the differential operator $\theta\theta'=\theta\circ\theta'$ (resp., $\theta f=\theta\circ f$) by $$\nabla_{\theta\theta'}:=\nabla_\theta\nabla_{\theta'}$$ (resp., $$\nabla_{\theta f}:=\nabla_\theta(f\,\cdot\,-))\;.$$ Since we thus define the action of an operator as the composite of the actions of the composing functions and vector fields, we get the compatibility condition \be\label{FC2}\nabla_{f\theta}=f\cdot\nabla_\theta\;,\ee and, as $\theta f=f\theta+\theta(f)$ (resp., $\theta \theta'=\theta'\theta+[\theta,\theta']$) -- where $\theta(f)$ (resp., $[\theta,\theta']$) denotes the Lie derivative $L_\theta f$ of $f$ with respect to $\theta$ (resp., the Lie bracket of the vector fields $\theta,\theta'$) -- , we also find the compatibility relations \be\label{FC3}\nabla_{\theta}(f\,\cdot\,-)=f\cdot \nabla_{\theta}+\theta(f)\,\cdot\,-\,\ee (resp., \be\label{FC4}\nabla_\theta\nabla_{\theta'}=\nabla_{\theta'}\nabla_\theta+\nabla_{[\theta,\theta']})\;.\ee In view of Equations (\ref{FC1}) -- (\ref{FC4}), {\it a $\cD$-module structure on $M\in{\tt Mod}(\cO)$ is the same as a flat connection on $M$}. Note that we implicitly worked out a left $\cD$-module structure. There is a similar result for right $\cD$-modules.\medskip

When resuming now our explanations given in Subsection \ref{D-ModulesAlgebras}, we understand that a morphism $\nabla$ of sheaves of $\K$-vector spaces satisfying the conditions (1) -- (3) is exactly a family of $\cD_X(U)$-modules $\cM_X(U)$, $U\in{\tt Open}_X$, such that the $\cD_X(U)$-actions are compatible with restrictions, i.e., is exactly a $\cD_X$-module structure on the considered sheaf $\cM_X$ of $\cO_X$-modules.\medskip

As concerns {\it examples}, it follows from what has been said that $\cO\in{\tt Mod}(\cD)$ with action $\nabla_\theta=L_\theta$, that top differential forms $\zW^{\op{top}}\in{\tt Mod}(\cD^{\op{op}})$ with action $\nabla_\theta=-L_\theta$, and that $\cD\in{\tt Mod}(\cD)\cap{\tt Mod}(\cD^{\op{op}})$ with action given by left and right compositions.

\subsection{Appendix 3 -- Sheaves versus global sections}\label{ShVsGlobSec}

In Classical Differential Geometry, the fundamental spaces (resp., operators), e.g., vector fields, differential forms... (resp., the Lie derivative, the de Rham differential...) are sheaves (resp., sheaf morphisms). Despite this sheaf-theoretic nature, most textbooks present Differential Geometry in terms of global sections and morphisms between them. Since these sections are sections of vector bundles (resp., these global morphisms are local operators), restriction and gluing is canonical (resp., the existence of smooth bump functions allows to localize the global morphisms in such a way that they commute with restrictions; e.g., for the de Rham differential, we have $$(d|_U\zw_U)|_V=\lp d(\za_V\zw_U)\rp|_V\quad{\text{and}}\quad d|_U\zw|_U=(d\zw)|_U\;,$$ where $\za_V$ is a bump function with constant value 1 in $V\subset U$ and support in $U$). Such global viewpoints are not possible in the real-analytic and holomorphic settings, since no interesting analytic bump functions do exist.\medskip

There is a number of well-known results on the equivalence of categories of sheaves and the corresponding categories of global sections, essentially when the topological space underlying the considered sheaves is an affine scheme or variety.\medskip

If $(X,\cO_X)$ is an affine algebraic variety, there is an equivalence \be\label{ShVsSectcMod} \zG(X,\bullet): {\tt cMod}(\cO_X)\rightleftarrows {\tt fMod}(\cO_X(X)):\cO_X\0_{\cO_X(X)}\bullet\;,\ee between the category ${\tt cMod}(\cO_X)$ of coherent $\cO_X$-modules and the category ${\tt fMod}(\cO_X(X))$ of finitely generated $\cO_X(X)$-modules \cite{Serre}. Similarly, for an affine scheme $(X,\cO_X)$, we have an equivalence \cite{Har} \be\label{ShVsSectqcMod}\zG(X,\bullet): {\tt qcMod}(\cO_X)\rightleftarrows {\tt Mod}(\cO_X(X)):\widetilde{\bullet}\ee between the category of quasi-coherent $\cO_X$-modules and the category of $\cO_X(X)$-modules. If $\cO_X(X)$ is Noetherian, the same functors define an equivalence between ${\tt cMod}(\cO_X)$ and ${\tt fMod}(\cO_X(X))$, just as in the case of an affine algebraic variety (the coordinate ring of an affine variety is Noetherian). It can easily be seen that the functors $\cO_X\0_{\cO_X(X)}\bullet$ and $\widetilde{\bullet}$ are isomorphic. \medskip

There exist similar equivalence results for locally free sheaves \cite{Serre}. Over an affine algebraic variety $X$ or an affine scheme $X=\op{Spec}R$, with $R$ Noetherian, the global section functor $\zG(X,\bullet)$ yields an equivalence of categories \be {\tt lfrMod}(\cO_X)\rightleftarrows {\tt pfMod}(\cO_X(X))\label{ShVsSectlfrMod}\;\ee between the category of locally free $\cO_X$-modules of finite rank and the category of projective finitely generated modules over $\cO_X(X)=R$.

Locally free sheaves of $\cO_X$-modules are viewed as algebraic vector bundles over $X$. However, many authors switch tacitly between locally free sheaves ${\tt lfrMod}(\cO_X)$ of $\cO_X$-modules of finite rank and finite rank vector bundles ${\tt rVB}_X$ over $X$, also in contexts where both concepts are independently defined. For a proof of this equivalence of categories over a premanifold $X$ (i.e., a smooth manifold that is not necessarily Hausdorff and second countable), see for instance \cite[Prop. 5.14]{Wed}.

In 1962, Swan proved a theorem similar to (\ref{ShVsSectlfrMod}) over a compact Hausdorff space $X$. The global continuous section functor induces an equivalence of categories \be {\tt rC^0VB}_X\rightleftarrows {\tt pfMod}(C^0_X(X))\label{Swan}\;\ee between the category of finite rank topological vector bundles on $X$ and the category of projective finitely generated modules over the continuous function algebra $C^0_X(X)$. In 2003, Nestruev \cite{Nes} extended this result to smooth manifolds: if $X$ is a smooth manifold, the global smooth section functor provides an equivalence of categories \be {\tt rC^\infty VB}_X\rightleftarrows {\tt pfMod}(C^\infty_X(X))\label{Swan}\;\ee between the category of finite rank smooth vector bundles on $X$ and the category of projective finitely generated modules over the smooth function algebra $C^\infty_X(X)$.

\subsection{Appendix 4 -- Model categories}\label{ModCat}


Model categories are a setting in which Homotopy Theory is well-developed. Homotopy Theory allows identifying, say, topological spaces, which, although not homeomorphic, still look similarly. Indeed, identification requirements that are weaker than homeomorphisms do exist. For instance, two {\bf homotopy equivalent} (resp., {\bf weakly homotopy equivalent}) spaces, i.e., two spaces related by two continuous maps that are inverses up to $C^0$-homotopies (resp., two spaces related by a continuous map that induces isomorphisms between all homotopy groups), are said to have the same homotopy type (resp., the same weak homotopy type): they may be considered as having the same basic shape.\medskip

Similar concepts exist in homological algebra. Since we study modules via their resolutions (chain complexes whose homology is the module under investigation), we are often not interested in the complex itself, but rather in its homology. Hence, we study complexes up to {\bf quasi-isomorphisms} (chain maps that induce an isomorphism in homology), or up to {\bf homotopy equivalences} (two chain maps that are inverses up to chain homotopies). In the homotopy category, we identify homotopic chain maps, so that homotopy equivalent chain complexes become isomorphic in the homotopy category. In the derived category, we then still quotien out the quasi-isomorphisms. The quasi-isomorphisms of homological algebra correspond to the weak homotopy equivalences of topology.\medskip

The notions of weak homotopy equivalence or quasi-isomorphism underly the axiomatic definition of a {\bf model category}. A model category, by definition, contains a class of morphisms called {\bf weak equivalences}, and these morphisms become isomorphisms upon passing to the associated homotopy category.\medskip

Let us first recall that a {\bf functorial factorization} in a category $\tc$ is a pair $(F,G)$ of endofunctors of the category ${\tt Map\,C}$ of maps of $\tt C$, such that any $f\in {\tt Map\,C}$ reads $f=G(f)\circ F(f)\,$. Hence, if $f:X\to Y$, we have $$X\stackrel{F(f)}{\longrightarrow} Z\stackrel{G(f)}{\longrightarrow} Y\;.$$ The action of the functor $F$ (resp., $G$) on a morphism $(u,v)$ between $f:X\to Y$ and $g:X'\to Y'$, i.e., on a commutative square
$$
\xymatrix{
X\ar[d]^{u} \ar[r]^{f}&Y\ar[d]^{v\;\;\;\;,}\\
X'\ar[r]^{g}&Y'\\
}
$$
is a morphism between $F(f)$ (resp., $G(f)$) and $F(g)$ (resp., $G(g)$). By functorial factorization we mean that there is a commutative diagram
$$
\xymatrix{X \ar[r]^{F(f)} \ar[d]_{u} & Z\ar[r]^{G(f)} \ar[d]^{w}&Y \ar[d]^{v\;\;\;\;.}\\
X'\ar[r]^{F(g)}& Z'\ar[r]^{G(g)}&Y'\\
}
$$

\begin{defi} A {\bf model category} is a category ${\tt M}$ together with three classes of morphisms, {\bf weak equivalences} (weq-s for short), {\bf fibrations}, and {\bf cofibrations}, and with two functorial factorizations $(\za,\zb)$ and $(\za',\zb')$, which satisfy the following axioms:
\begin{itemize}
\item {\bf MC1} (Limit axiom). The category ${\tt M}$ is closed under small limits and colimits.
\item {\bf MC2} (Retract axiom). The three classes of morphisms are closed under retracts. More precisely, if $f$ is a retract of $g$, i.e., if we have a commutative diagram as in Figure 1,
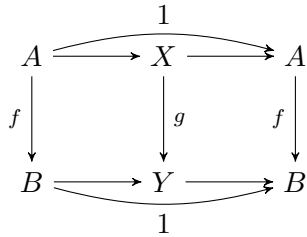
\begin{figure}[h]
\begin{center}
\begin{tikzpicture}
  \matrix (m) [matrix of math nodes, row sep=3em, column sep=3em]
    {  A & X & A \\
       B & Y & B \\ };
 \path[->]
 (m-1-1) edge node[below] {}  (m-1-2)
 (m-1-2) edge node[below] {} (m-1-3)
 (m-1-1) edge[bend left=15] node[above] {1} (m-1-3)
 (m-2-1) edge node[auto] {} (m-2-2)
 (m-2-2) edge node[auto] {} (m-2-3)
 (m-2-1) edge[bend right=15] node[below] {1} (m-2-3)
 (m-1-1) edge  node[left] {$\scriptstyle{f}$} (m-2-1)
 (m-1-2) edge  node[right] {$\scriptstyle{g}$} (m-2-2)
 (m-1-3) edge  node[left] {$\scriptstyle{f}$} (m-2-3);

\end{tikzpicture}\label{fig retract}\caption{Retract diagram}
\end{center}
\end{figure} and if $g$ belongs to one of the morphism classes, then $f$ belongs to the same class.

\item {\bf MC3} (2 out of 3 axiom). If $f$ and $g$ are two composable morphisms, and two of the maps $f,g$, and $g\circ f$ are weak equivalences, then so is the third.
\item {\bf MC4} (Lifting axiom). In a commutative diagram as in Figure 2, where
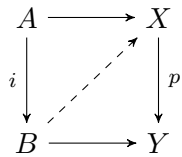
\begin{figure}[h]
\begin{center}
\begin{tikzpicture}
  \matrix (m) [matrix of math nodes, row sep=3em, column sep=3em]
    {  A & X  \\
       B & Y  \\ };
 \path[->,dashed]
 (m-2-1) edge  node[right] {$\scriptstyle$} (m-1-2);
 \path[->]
 (m-1-1) edge  (m-1-2);
 \path[->]
 (m-1-1) edge  node[left] {$\scriptstyle{i}$} (m-2-1);
 \path[->]
 (m-1-2) edge  node[right] {$\scriptstyle{p}$} (m-2-2);
 \path[->]
 (m-2-1) edge  (m-2-2);
\end{tikzpicture}\caption{Lifting diagram}\label{fig_lifting}
\end{center}
\end{figure} $i$ is a cofibration and $p$ a fibration, the lifting exists, if either $i$ or $p$ is trivial (a {\bf trivial (co)fibration} is a (co)fibration that is also a weak equivalence).

\item {\bf MC5} (Factorization axiom). For any morphism $f$, $\za(f)$ is a cofibration, $\zb(f)$ is a trivial fibration, $\za'(f)$ is a trivial cofibration, and $\zb'(f)$ is a fibration.
\end{itemize}
\end{defi}

Several comments are necessary.

Various definitions of model categories can be found in the literature. The above one is used by Hovey \cite{Hov}.

In fact, model categories were introduced by Quillen under the name of closed model categories. Quillen initially assumed only that all finite limits and colimits exist in a model category. Dwyer-Spalinski \cite{DS} use this axiom, Gel'fand-Manin \cite{GM} even assume only the existence of finite projective and inductive limits. The existence of all {\bf small limits and colimits} is required for instance in the texts of Goerss-Schemmerhorn \cite{GS} and Hovey.

Moreover, the factorization axiom {\small MC5} asks that any map can be factored into a cofibration followed by a trivial fibration, and into a trivial cofibration followed by a fibration. These factorizations are not always required to be functorial (see for instance Dwyer-Spalinski and Goerss-Schemmerhorn). However, if the model category is cofibrantly generated (see below), the `cofibration -- trivial fibration' and `trivial cofibration -- fibration' factorizations can be constructed by means of the small object argument (see below), and these factorizations are functorial. Hence, in most model categories, {\bf functorial factorizations} do exist. Hovey's axiom {\small MC5} not only asks that there exist two functorial factorizations, but assumes that a {\bf choice} has been made: his model categories come equipped with two such functorial factorizations.\medskip

Since a model category {\tt M} has all small limits and colimits, it has a terminal object $\star$ and an initial object $\emptyset$. As mentioned, a model category does not only contain weq-s $\stackrel{\sim}{\rightarrow}$, but has additional structure: cofibrations $\rightarrowtail$ (`good injections') and fibrations $\twoheadrightarrow$ (`good surjections'). Further, there are cofibrant and fibrant objects (`good objects'). We say that $Y\in {\tt M}$ is a {\bf cofibrant object} (resp., a {\bf fibrant object}), if the unique map $i_Y:\emptyset \to Y$ (resp., $p_Y:Y\to \star$) is a cofibration (resp., a fibration). Let now $Y\in{\tt M}$ be any object, i.e., not necessarily cofibrant or fibrant. If we apply the functorial factorization $(\za,\zb)$ to $i_Y$, we get a cofibration $\za(i_Y):\emptyset \rightarrowtail QY$ and a trivial fibration $q_Y:=\zb(i_Y):QY\stackrel{\sim}{\twoheadrightarrow} Y$. We refer to $QY$ as a {\bf cofibrant replacement} of $Y$ (`resolution', `{\small CW}-replacement'). A {\bf fibrant replacement} is defined dually.

If we say that $Q$ is a functor, we claim that, for any arrow $f:Y\to Y'$, we get an arrow $Qf:QY\to QY'$, such that composition and the units are respected. If we say that $Q$ is a functorial replacement functor, we claim that the replacement $QY \stackrel{\sim}{\twoheadrightarrow}Y$ is functorial, i.e., that $Q$ is a functor and that the functors $Q$ and $\id$ are related by a natural transformation, i.e., that the following diagram commutes: $$\begin{array}{ccccccccc} & QY & \stackrel{q_Y}{\twoheadrightarrow} & Y & \\Qf  & \downarrow &  & \downarrow & f\\ & QY' & \stackrel{q_{Y'}}{\twoheadrightarrow} & Y' &\end{array}\;.$$ Since we are given a functorial `cofibration -- trivial fibration' factorization $(\za,\zb)$, we actually get such a {\bf functorial cofibrant replacement functor}. Indeed, any $f:Y\to Y'$ induces a commutative square
$$\begin{array}{ccccc}& \emptyset &\stackrel{i_Y}{\longrightarrow} & Y &\\ \op{id} & \downarrow & & \downarrow & f\\ & \emptyset & \stackrel{i_{Y'}}{\longrightarrow} & Y' &\end{array}\;,$$ in view of the definition of an initial object. This square implements a commutative diagram
\be\label{FunReplFun}\begin{array}{ccccccccc}& \emptyset &\stackrel{\za(i_Y)}{\rightarrowtail} && QY & \stackrel{q_Y=\zb(i_Y)}{\twoheadrightarrow} & Y & \\ \op{id} & \downarrow & & w&\downarrow &  & \downarrow & f\\ & \emptyset &\stackrel{\za(i_{Y'})}{\rightarrowtail}  && QY' & \stackrel{q_{Y'}=\zb(i_{Y'})}{\twoheadrightarrow} & Y' &\end{array}\;.\ee We set $Qf:= w$, so that $Q$ becomes an endofunctor. As the {\small RHS} square in the last diagram commutes, $q$ is actually a natural transformation between $Q$ and $\id$, and $Q$ is a functorial cofibrant replacement functor. The {\bf functorial fibrant replacement functor} $R$ is defined similarly from $(\za',\zb')$.\medskip

Let us still mention three basic results on model categories that will be used implicitly.

\begin{enumerate}
\item Two of the distinguished classes determine the third.
\item All three distinguished classes are closed under composition.
\item A functorial cofibrant or fibrant replacement functor respects weq-s.
\end{enumerate}

Indeed, it can be shown that the class $\op{Cof}$ of all cofibrations is not only contained in the class $\op{LLP(TrivFib)}$ of those maps that have the left lifting property with respect to all trivial fibrations, but that \be\label{Cof1}\op{Cof}=\op{LLP(TrivFib)}\;.\ee Similarly weq-s and cofibrations determine fibrations: \be\label{Fib1}\op{Fib}=\op{RLP(TrivCof)}\;.\ee Moreover, \be\label{TrivCofFib1}\op{TrivCof}=\op{LLP(Fib)}\quad\text{and}\quad\op{TrivFib}=\op{RLP(Cof)}\;.\ee
The 2 out of 3 axiom shows that $f$ is a weq if and only if $\za(f)$ is a trivial cofibration, or, if and only if $\zb'(f)$ is a trivial fibration. Hence, cofibrations and fibrations determine weq-s.\medskip

It follows from the characterizations (\ref{Cof1}), (\ref{Fib1}), and (\ref{TrivCofFib1}) that all three distinguished classes are closed under composition. Further, Diagram (\ref{FunReplFun}) shows that the functorial cofibrant replacement functor $Q$ transforms a weq into a weq.

\subsection{Appendix 5 -- Smallness}\label{Small}

In the following, we use the {\bf axiom of choice}, which claims that, for any family $(S_i)_{i\in I}$ of non-empty sets, we can choose a unique element $(s_i)_{i\in I}$ in each set. Although there is no proof of the existence in whole generality of a choice function $f(S_i)=s_i$, for all $i\in I\,,$ the axiom of choice is accepted by a majority of mathematicians.

\subsubsection{Ordinals}

A {\bf well-ordered set} is totally ordered set that does not contain any infinite decreasing sequence. An {\bf order-isomorphism} between well-ordered sets $W_1,W_2$ (the notion can even be defined between partially ordered sets) is a bijection $f:W_1\to W_2$ such that $x\le y$ if and only if $f(x)\le f(y)$. This condition is equivalent to asking that $f$ and $f^{-1}$ be strictly increasing. Being order-isomorphic is an equivalence $\cI$ in well-ordered sets. The $\cI$-equivalence classes are the {\bf ordinal numbers}. More precisely, each well-ordered set is order-isomorphic to a unique ordinal number, which is then identified with its equivalence class.\medskip

The finite ordinals are the non-negative integers. An ordinal can be viewed as the well-ordered set of all (strictly) smaller ordinals: \be\label{FinOrd}0=\emptyset, 1=\{0\}, \ldots, n=\{0<1<\ldots<n-1\}, \ldots, \zw=\{0<1<\ldots\}\;.\ee The ordinal $\zw$ is the first non-finite ordinal. The next non-finite ordinals are \be\label{InfOrd}\zw, \zw + 1, \zw + 2, \ldots, \zw\cdot 2, \zw\cdot 2 + 1,\ldots, \zw^2, \ldots, \zw^3, \ldots, \zw^\zw, \ldots\;\ee We also often think of an ordinal as a category where there is a unique map from $\za$ to $\zb$ if and only if $α\za \le \zb\,$.\medskip

In the list (\ref{FinOrd}), $\zw$ is a limit ordinal and all other non-zero ordinals are successor ordinals. Actually, any ordinal is either {\bf zero}, or a {\bf successor ordinal}, or a {\bf limit ordinal}. More precisely, a {\bf limit ordinal} is an ordinal $\za$ such that there exists an ordinal $\zb<\za$, and whenever $\zb<\za$ there exists an ordinal $\zg$ such that $\zb<\zg<\za$. Although not every ordinal is a successor, every ordinal has a successor.

\subsubsection{Cardinals}

The {\bf well-ordering theorem}, which is equivalent to the axiom of choice, states that every set can be well-ordered. \medskip

Two sets $S_1,S_2$ are {\bf equinumerous} if there exists a bijection $f:S_1\to S_2$ between them. Being equinumerous is an equivalence $\cE$ in the class of sets. The {\bf cardinality of a set} is its $\cE$-equivalence class. For any given set, we can choose well-orderings, thus obtaining well-ordered sets. The latter are equinumerous since they have the same underlying set. Each one of these well-ordered sets is order-isomorphic to a unique ordinal. However, these ordinals, although equinumerous and thus of same cardinality, are not necessarily order-isomorphic, i.e., they are potentially different. We identify the $\cE$-equivalence class of a set with the smallest ordinal in this class: the cardinality of a set thus coincides with this smallest ordinal. The cardinalities of sets are the {\bf cardinal numbers}.\medskip

The finite cardinals are the cardinalities of the finite sets, i.e., they are the non-negative integers. The first non-finite cardinal is the cardinality $\aleph_0$ of the set $\N$ of natural numbers (infinite countable set). The cardinality of the set $\R$ of real numbers (infinite uncountable set) is $2^{\aleph_0}$, i.e., it is equal to the cardinality of the set of all subsets of $\N$. The {\bf continuum hypothesis} says that there is no set whose cardinality $\aleph_1$ is strictly between the cardinality $\aleph_0$ of the integers and the cardinality $2^{\aleph_0}$ of the real numbers: $\aleph_1 = 2^{\aleph_0}$. The non-finite or transfinite cardinals are denoted by \be\label{InfCard}\aleph_0, \aleph_1, \aleph_2,\ldots\;\ee As mentioned above, different ordinals may have the same cardinality. For instance, all the ordinals of (\ref{InfOrd}) (uncountably many countably infinite ordinals) have cardinality $\aleph_0$ (unique countably infinite cardinal). According to what we said above, we identify $\aleph_0$ with $\zw$. Further, the set of all countable ordinals constitutes the first uncountable ordinal $\zw_1$, which we identify with $\aleph_1$.

\subsubsection{Filtration of an ordinal with respect to a cardinal}

This concept is the key of all smallness issues that are considered in the following. An {\bf ordinal $\zl$ is filtered with respect to a cardinal $\zk$} (we say also that $\zl$ is $\zk$-filtered), if $\zl$ is a limit ordinal and \be\label{KeySmall}A\subset \zl\;\text{and}\;|A|\le \zk\quad \Rightarrow\quad \sup A < \zl\;,\ee i.e., the supremum of a subset of $\zl$ of cardinality at most $\zk$ is smaller than $\zl$. For instance, let $\zk$ be the cardinal $\zw=\aleph_0$, and let $\zl$ be the limit ordinal $\zw\cdot 2=\{0,1,\ldots,\zw,\zw+1,\ldots\}$. If $A=\{\zw,\zw+1,\ldots\}$, the assumptions are satisfied, but $\sup A=\zw\cdot 2=\zl$, so that $\zl$ is not $\zk$-filtered. {\it The condition `$\zl$ is $\zk$-filtered' is actually a largeness condition for $\zl$ with respect to $\zk$}. {\bf If $\zl$ is $\zk$-filtered for $\zk>\zk'$, then $\zl$ is also $\zk'$-filtered}. For an {\it infinite cardinal} $\zk$ (e.g., for the above considered $\zk=\zw=\aleph_0$), the smallest $\zk$-filtered ordinal is the first cardinal $\zk_1$ larger than $\zk$ (in our example $\zk_1=\aleph_1$; actually any successor cardinal is $\zk$-filtered, but there are also non-cardinal ordinals that are $\zk$-filtered, e.g., $\zk_1\cdot 2$). For a {\it finite cardinal} $\zk$, a $\zk$-filtered ordinal is just a limit ordinal.

\subsubsection{Small objects}

The definition varies from author to author. We stick to the definition of \cite{Hov} and compare it with the definition of \cite{GS}.\medskip

Smallness of an object $A$ is defined with respect to a category $\tt C$ (assumed to have all small colimits), a class of morphisms $W$ in $\tt C$, and a cardinal $\zk$ (that can depend on $A$). The point is that the covariant Hom-functor $${\tt C}(A,\bullet):=\op{Hom}_{\tt C}(A,\bullet)$$ commutes with limits, but usually not with colimits. However, {\it if the considered sequence is sufficiently large with respect to $A$, then commutation may be proven}. More precisely, if $A\in\tt C\,,$ we consider the colimits of all the sequences that are large enough with respect to some cardinal $\zk$ (possibly $\zk(A)$) and have their arrows in $W$, i.e., the colimits of all the $\zl$-sequences with arrows in $W$ for all $\zk$-filtered ordinals $\zl$, and try to prove that the covariant Hom-functor ${\tt C}(A,\bullet)$ commutes with these colimits. In this case, we say that {\bf $A\in\tt C$ is small with respect to $\zk$ and $W$}.\medskip

To provide deeper understanding, we show that any set $A\in\tt Set$ is small with respect to $\zk=|A|$ and all morphisms in $\tt Set$. It suffices to prove that, for any $\zk$-filtered ordinal $\zl$ and any $\zl$-sequence $(X_\zb)_{\zb<\zl}$, the canonical map $$c:\op{colim}_{\zb<\zl}{\tt Set}(A,X_\zb)\ni [f_\zb] \mapsto \zf_\zb\circ f_\zb\in{\tt Set}(A,{\op{colim}_{\zb<\zl}X_\zb)}\;$$ is an isomorphism. Let us recall that a $\zl$-sequence in $\tt Set$ is a colimit respecting functor $\zl\to {\tt Set}$, so that we actually deal with a direct system and a direct limit. This direct limit $\colim_{\zb<\zl}X_\zb$ has a simple construction $\coprod_{\zb<\zl} X_\zb/\sim$, where notation is self-explaining, and so has the {\small LHS} colimit. The canonical map $c$ -- whose existence is guaranteed by the universality property of a colimit -- is then obtained as indicated, where $\zf_\zb$ denotes the map from $X_\zb$ to the direct limit.

Let now $f\in{\tt Set}(A,\op{colim}_{\zb<\zl}X_\zb)$. For any $a\in A$, $f(a)$ is a class that has a representative in some $X_{\zb(a)}$. The idea is to find an $X_\zg$ that contains all the `images'. We have $S:=\{\zb(a):a\in A\}\subset \zl$ and $|S|\le|A|=\zk$. Hence, $\zg:=\sup S<\zl$ and, for any $a$, $f(a)$ is a class with representative in $X_\zg$: $f$ defines a map $\tilde{f}\in{\tt Set}(A,X_\zg)$ and a class $[\tilde f]\in\op{colim}_{\zg<\zl}{\tt Set}(A,X_\zg)$, whose image by $c$ is $\zf_\zg\circ\tilde{f}=f$. It now suffices to prove that $c$ is not only surjective, but also injective. This proof is similar.\medskip

Note that `$A\in\tt C$ is $\zk$-small with respect to $W$' roughly means that a map from $A$ to the colimit of a sequence with arrows in $W$ that is indexed by a $\zk$-filtered ordinal, or, better, from $A$ to a sufficiently long composition of arrows in $W$, factors through some stage of this composition. Further, we were able to prove commutation of ${\tt Set}(A,\bullet)$ with $\zl$-colimits, only because we were allowed to choose $\zl$ sufficiently large with respect to $\zk=|A|\,.$ It follows that, {\bf if $\zk<\zk'$, then $\zk$-smallness implies $\zk'$-smallness} ($\zk'$-filtered ordinals are also $\zk$-filtered).\medskip

Eventually, we say that $A\in\tt C$ is {\bf small relative to $W$}, if it is $\zk$-small relative to $W$ for some cardinal $\zk$, and we say that $A$ is {\bf small}, if it is small relative to all the morphisms of the underlying category $\tt C$. An object in ${\tt C}\cap {\tt D}$ can be small in $\tt C$ but not in $\tt D$.\medskip

In \cite{GS}, `small' (with respect to $W$) means `sequentially small': the covariant Hom-functor commutes with the colimits of the $\zw$-sequences. This requirement is analogous to `$n$-small', i.e., small relative to a finite cardinal $n\in\N$: the covariant Hom-functor commutes with the colimits of the $\zl$-sequences for all limit ordinals $\zl$. In \cite{Hov}, `small' (relative to $W$) means, as just mentioned, $\zk$-small for some $\zk$: the covariant Hom-functor commutes with the colimits of all the $\zl$-sequences for all the $\zk$-filtered ordinals $\zl$. In view of what has been said above, it is clear that $n$-small implies $\zk$-small, for any $\zk>n$.

\subsection{Appendix 6 -- Cofibrantly generated model categories}\label{CofGenModCat}

We give two versions of the definition of a cofibrantly generated model category, a finite definition \cite{GS} and a transfinite \cite{Hov} one.

\subsubsection{Finite definition}

A model category is {\bf cofibrantly generated} \cite{GS}, if there exist {\it sets} of morphisms $I$ and $J$, which generate the cofibrations and the trivial cofibrations, respectively, i.e., more precisely, if there are sets $I$ and $J$ such that

\begin{enumerate}

\item the source of every morphism in $I$ is sequentially small with respect to the class $\op{Cof}$ and $\op{TrivFib}=\op{RLP}(I)\,$,

\item the source of every morphism in $J$ is sequentially small with respect to the class $\op{TrivCof}$ and $\op{Fib}=\op{RLP}(J)\,$.

\end{enumerate}
It then follows that $I$ and $J$ are actually the generating cofibrations and the generating trivial cofibrations: $$\op{Cof}=\op{LLP}(\op{RLP}(I))\quad\text{and}\quad\op{TrivCof}=\op{LLP}(\op{RLP}(J))\;.$$

\subsubsection{Transfinite composition of pushouts}

The transfinite definition of a cofibrantly generated model category contains a transfinite smallness condition, namely that the domains of the maps in $I$ and $J$ are small ($\zk$-small for some fixed $\zk$) relative to ``transfinite compositions of pushouts of arrows in $I$ and $J$'', respectively.\medskip

We briefly explain this smallness requirement \cite{Hov}. Let $\tt C$ be a category that has all small colimits, let $\zl$ be an ordinal, and let $X:\zl\to \tt C$ be a {\bf $\zl$-sequence in $\tt C$}, i.e., a colimit respecting functor from the category $\zl$ to the category $\tt C$. Usually this diagram is denoted by $X_0\to X_1\to\ldots\to X_\zb\to \ldots$ It is natural to refer to the map $$X_0\to \op{colim}_{\zb<\zl}X_\zb$$ as {\bf the composite of the $\zl$-sequence} $X$. If $W$ is a class of morphisms in $\tt C$ (in the following often the set $I$ or the set $J$) and every map $X_\zb\to X_{\zb+1}$, $\zb+1<\zl$, is in $W$, we refer to the composite $X_0\to \op{colim}_{\zb<\zl}X_\zb$ as a {\bf transfinite composition of maps in $W$}.\medskip

Let us also recall that, if we have a commutative square in $\tt C$, the right down arrow is said to be the {\bf pushout} of the left down arrow.\medskip

We now describe a subclass of $\op{LLP}(\op{RLP}(W))$ that will be denoted by $W$-cell. An element of {\bf $W$-cell} is a {\bf transfinite composition of pushouts of arrows in $W$}. In other words, a $\tt C$-map $f: A\to B$ is in $W$-cell, if there is an ordinal $\zl$ and a $\zl$-sequence $X:\zl\to \tt C$ such that $f$ is the composite of $X$ and such that, for each $\zb+1<\zl$, there is a pushout square in which the right down arrow is $X_\zb\to X_{\zb+1}$ and the left down arrow belongs to $W$.

\subsubsection{Transfinite definition}

A model category is {\bf cofibrantly generated} \cite{Hov}, if there exist sets $I$ and $J$ of maps such that

\begin{enumerate}
\item the domains of the maps in $I$ are small ($\zk$-small for some fixed $\zk$) relative to $I$-cell,

\item the domains of the maps in $J$ are small ($\zk$-small for some fixed $\zk$) relative to $J$-cell,

\item $\op{TrivFib}=\op{RLP}(I)\,$, and

\item $\op{Fib}=\op{RLP}(J)\,$.
\end{enumerate}

The condition on the domains of the maps $I\ni i:A\to B$ (resp., $J\ni j:A\to B$) is that there is a cardinal $\zk$, such that ${\tt C}(A,\bullet)$ commutes with the colimits of all the sequences, which are indexed by a $\zk$-filtered ordinal, and whose arrows are in $I$-cell (resp., $J$-cell), i.e., are transfinite compositions of pushouts of arrows in $I$ (resp., $J$).


\subsubsection{Small object argument}

The small object argument is a functorial construction of a factorization system (e.g., into a cofibration and a trivial fibration, or, into a trivial cofibration and a fibration).\medskip

\begin{theo}[Small Object Argument] Let $\tt C$ be a category with all small colimits and let $W$ be a set of $\tt C$-maps. If the domains of the maps in $W$ are small $(\,$$\zk$-small for some fixed $\zk$$\,)$ with respect to $W$\text{-\em cell}, there exists a functorial factorization $(\za,\zb)$, such that, for any $\tt C$-map $f:A\to B$, the map $\za(f):A\to C$ $(\,$resp., $\zb(f):C\to B$$\,)$ is in $$W\text{-\em cell}\subset \op{LLP}(\op{RLP}(W))$$ $(\,$resp., in $\op{RLP}(W)$$\,)$.\label{SOA}\end{theo}

If $\tt C$ is a cofibrantly generated model category in the sense of \cite{Hov} and $W=I$ (resp., $W=J$), the theorem shows that there is a functorial factorization $(\za,\zb)$ (resp., $(\za',\zb')$), such that any map $f$ in $\tt C$ factors into a map $\za(f)\in I\text{-cell}\subset \op{Cof}$ and a map $\zb(f)\in \op{TrivFib}$ $(\,$resp., into a map $\za'(f)\in J\text{-cell}\subset\op{TrivCof}$ and a map $\zb'(f)\in \op{Fib}$$\,)$.\medskip

In fact, the functorial factorization of Theorem \ref{SOA} is constructed via a possibly transfinite induction. It turns out that it suffices that the covariant Hom-functors ${\tt C}(A,\bullet)$, obtained from the sources $A$ of the maps in $W$, commutate with the colimits of the $\zl$-sequences, for some $\zk$-filtered ordinal $\zl$. In the case $\zk=n\in\N$, it is customary to choose the smallest $\zl$, i.e., $\zl=\zw$. In other words, it is enough to assume that the sources $A$ be sequentially small. This justifies the finite definition of a cofibrantly generated model category \cite{GS}.\medskip

An excellent explanation of the Small Object Argument can be found in \cite{DS}.\medskip

It is clear that the finite definition \cite{GS} is stronger than the transfinite one \cite{Hov}. First, $n$-smallness implies $\zk$-smallness, and, second, smallness with respect to $\op{Cof}$ (resp., $\op{TrivCof}$) implies smallness with respect to $I$-cell (resp., $J$-cell).\medskip

The model structures we study in the present paper will all be {\it finitely} generated. A {\bf finitely generated model structure} is a cofibrantly generated model structure \cite{Hov}, such that $I$ and $J$ can be chosen so that their sources and targets are $n$-small, $n\in\N$, relative to $\op{Cof}$. This implies in particular that our model structures are cofibrantly generated in the sense of \cite{GS}.

\subsection{Appendix 7 -- Invariants versus coinvariants}\label{InvCoinv}

If $G$ is a (multiplicative) group and $k$ a commutative unital ring, we denote by $k[G]$ the group $k$-algebra of $G$ (the free $k$-module made of all formal finite linear combinations $\sum_{g\in G} r(g)\, g$ with coefficients in $k$, endowed with the unital ring multiplication that extends the group multiplication by linearity).\medskip

In the following, we use notation of Subsection \ref{Adjunction}. Observe that $\0^n_\cO M_\bullet$ is a module over the group $\cO$-algebra $\cO[\mathbb{S}_n]$, where $\mathbb{S}_n$ denotes the $n$-th symmetric group. There is an $\cO$-module isomorphism \be\label{Iso1}{\cal S}_\cO^n M_\bullet=\0_\cO^n M_\bullet/{\cal I}\cap \0_\cO^n M_\bullet\simeq (\0_\cO^n M_\bullet)_{\mathbb{S}_n}:=\0_\cO^n M_\bullet/\langle T-\zs\cdot T\ra\;,\ee where $(\0_\cO^n M_\bullet)_{\mathbb{S}_n}$ is the $\cO$-module of $\mathbb{S}_n$-coinvariants and where the denominator is the $\cO$-submodule generated by the elements of the type $T-\zs\cdot T$, $T\in\0_\cO^n M_\bullet$, $\zs\in\mathbb{S}_n$ (a Koszul sign is incorporated in the action of $\zs$). It is known that, since the cardinality of $\mathbb{S}_n$ is invertible in $\cO$, we have also an $\cO$-module isomorphism \be\label{Iso2}(\0_\cO^n M_\bullet)_{\mathbb{S}_n}\simeq(\0_\cO^n M_\bullet)^{\mathbb{S}_n}:=\{T\in \0_\cO^n M_\bullet: \zs\cdot T=T,\forall \zs\in\mathbb{S}_n\}\;\ee between the $\mathbb{S}_n$-coinvariants and the $\mathbb{S}_n$-invariants. The averaging map or graded symmetrization operator \be\label{SymOp}{\frak S}:\0_\cO^n M_\bullet\ni T\mapsto \frac{1}{n!}\sum_{\zs\in\mathbb{S}_n}\zs\cdot T\in(\0_\cO^n M_\bullet)^{\mathbb{S}_n}\;\ee coincides with identity on $(\0_\cO^n M_\bullet)^{\mathbb{S}_n}$, what implies that it is surjective. When viewed as defined on coinvariants $(\0_\cO^n M_\bullet)_{\mathbb{S}_n}\,$, it provides the mentioned isomorphism (\ref{Iso2}). It is straightforwardly checked that the graded symmetric multiplication $\vee$ on $(\0_\cO^\ast M_\bullet)^{\mathbb{S}_\ast}$, defined by \be\label{Mult2}{\frak S}(S)\vee {\frak S}(T)={\frak S}({\frak S}(S)\0 {\frak S}(T))\;,\ee endows $(\0_\cO^\ast M_\bullet)^{\mathbb{S}_\ast}$ with a {\small DG} $\cD$-algebra structure, and that the $\cO$-module isomorphism \be\label{Alg2}{\cal S}_\cO^\ast M_\bullet\simeq \{T\in \0_\cO^\ast M_\bullet: \zs\cdot T=T,\forall \zs\in\mathbb{S}_\ast\}\ee is in fact a $\tt DG\cD A$-isomorphism.

\subsection{Appendix 8 -- Relative Sullivan algebras}\label{RSKA}

This Appendix contains background information on Rational Homotopy Theory and on relative Sullivan algebras.

\subsubsection{Rational Homotopy Theory}\label{RHT}

The idea of Homotopy Theory is to identify two topological spaces when they have the same basic shape although they may not be homeomorphic. For instance, the first homotopy group is known to encode information about holes. Hence, the suggestion to identify spaces that are weakly homotopy equivalent, i.e., that are related by a continuous map that induces isomorphisms between all homotopy groups.\medskip

Since the difficulty with homotopy groups is torsion, it seems natural to try to eliminate torsion elements. We confine ourselves to simply-connected spaces, i.e., path-connected topological spaces with trivial first homotopy group. Since the higher homotopy groups $\zp_n(X)$, $n>1$, of simply-connected spaces $X$ are Abelian, all homotopy groups $\zp_n(X)$, $n\ge 1$, are $\Z$-modules. Their tensor products $\Q\0_\Z\zp_n(X)$ over $\Z$ with the field $\Q$ of rational numbers are $\Q$-vector spaces. The point is that the natural inclusion $$\zp_n(X)\ni c\mapsto 1\0 c\in \Q\0_\Z \zp_n(X)$$ sends torsion elements to zero. Indeed, if, in $\zp_n(X)$, we have $zc=0$ for some $z\in\Z\setminus\{0\}$, then, in the $\Q$-vector space $\Q\0_\Z \zp_n(X)$, we get $$0=1\0 zc=z\0 c=z(1\0 c)\;,$$ so that $1\0 c=0.$ This observation justifies the replacement of the homotopy groups $\zp_n(X)$ by the $\Q$-vector spaces $\Q\0_\Z\zp_n(X)$. In Rational Homotopy Theory, we then identify simply-connected topological spaces that are related by a weak rational homotopy equivalence, i.e., by a continuous map that induces isomorphisms between all rationalized homotopy groups $\Q\0_\Z\zp_\star(-)\,$.

\subsubsection{Algebraic models}

In view of Section \ref{RHT}, the important category is the homotopy category. The categories of topological spaces and simplicial sets are known to have equivalent homotopy categories. Hence, simplicial sets are purely combinatorial models of the homotopy classes of topological spaces. Kan (1958) constructed algebro-combinatorial models for classical Homotopy Theory: simplicial groups. There exists a spectral sequence for homotopy groups of a simplicial group that starts with the homotopy groups of a simplicial {\small LA}. In Rational Homotopy Theory, {\small DGLA}-s and {\small DGCA}-s were used as models: Quillen (1969) (resp., Sullivan (1977)) associated to certain simply-connected topological spaces a {\small DGLA} (resp., a {\small DGCA}) over $\Q$, which knows about the rational homotopy type of the space.\medskip

More precisely, a simply-connected topological space is called rational, if its homotopy groups are not only $\Z$-modules but $\Q$-vector spaces. It turns out that for any simply-connected space $X$ there is a continuous map $f:X\to X_\Q\,$, whose target is a simply-connected rational space, and which induces isomorphisms between all rationalized homotopy groups. Hence, from the standpoint of Rational Homotopy Theory, it is enough to study simply-connected rational spaces.\medskip

What Quillen actually proved in 1969 is that the homotopy categories of simply-connected rational topological spaces and of connected {\small DGLA}-s over $\Q$ are equivalent.\medskip

Similarly, Sullivan showed in 1977 that there exists a categorical equivalence between the homotopy categories of simply-connected rational topological spaces with finite Betti numbers and of {\small DGCA}-s $(A^\bullet,d)$ over $\Q$, whose cohomology spaces satisfy $H^0(A^\bullet,d)=\Q$, $H^1(A^\bullet,d)=0$, and $H^n(A^\bullet,d)$ is finite-dimensional for any $n$. This equivalence is implemented by an adjoint pair. The left adjoint assigns to any space $X$ of the underlying source category, a {\small DGCA} $(A^\bullet(X),d)$ in the underlying target category. The model $A^{\bullet}(X)$ is usually large, but can be replaced by a smaller one: there is a quasi-isomorphism of {\small DGCA}-s $\zf:(\w V^\bullet,d)\to (A^\bullet(X),d)$, where $\w V^\bullet$ is the free {\small GC} $\Q$-algebra over a graded $\Q$-vector space $V^\bullet=\bigoplus_{i}V^i$ and where $d V^\bullet\subset \w^{\ge 2}V^\bullet$. The {\small DGCA} $(\w V^\bullet,d)$, which is unique up to isomorphism, is the Sullivan minimal model of $X$.

\subsubsection{Sullivan algebras}

A {\bf Sullivan algebra} is a quasi-free {\small DGCA} $(\w V^\bullet,d)$ over a {\small GV} (graded vector space) $V^\bullet$ (over $\Q$ or any other field $\K$ of characteristic zero), whose differential $d$ satisfies some lowering condition. Quasi-free means free as {\small GCA} but not necessarily free as {\small DGCA}. Recall that in Remark \ref{FreeDGDA}, we emphasized that $\cS M_\bullet$ is the free {\small DGCA} over the {\small DGM} (differential graded module) $M_\bullet\,$. The point is that in the case of $\w V^\bullet$, the space $V^\bullet$ is not a {\small DGV}. On the contrary, the mentioned lowering condition requires roughly that $V^\bullet$ be endowed with an increasing filtration $V(\bullet)$, such that $dV(k)\subset \w V(k-1)\,$.\medskip

A Sullivan model of a {\it \small DGCA} $(A^\bullet, d)$ is, as above, a quasi-isomorphism $\zf:(\w V^\bullet,d)\to(A^\bullet,d)$ of {\small DGCA}-s from a Sullivan algebra $(\w V^\bullet,d)$ to $(A^\bullet,d)\,$. {\it Morphisms of {\small DGCA}-s} $f:(B^\bullet,d)\to (A^\bullet,d)$ are modeled by relative Sullivan algebras $(B^\bullet\0 \w V^\bullet,d)\,.$ {\bf Relative Sullivan algebras} generalize Sullivan algebras. The key point for relative Sullivan algebras is similar to the one for Sullivan algebras. Just as in the latter case $d$ does not stabilize $V^\bullet$, it does in the former not stabilize $\w V^\bullet$: whereas $B^\bullet$ is always a sub-{\small DGCA} of the relative Sullivan algebra $(B^\bullet\0\w V^\bullet,d)$, the factor $\w V^\bullet$ is usually not. In particular, a relative Sullivan algebra is mostly not the tensor product of two {\small DGCA}-s.\medskip

Details on relative Sullivan algebras over a field can be found in \cite{FHT} and \cite{Hes}. In Subsection \ref{RSDA} of the present text, we define relative Sullivan algebras over the ring of differential operators.

\end{document}